\documentclass[a4paper,reqno,11pt]{amsart}
\usepackage{mathrsfs}
\usepackage[all]{xy}
\usepackage{txfonts}
\usepackage{amsfonts}
\usepackage{amssymb}
\usepackage{amsfonts,amsmath,amsthm,amssymb,stmaryrd}
\usepackage{hyperref}
\usepackage[numbers,sort&compress]{natbib}

\newtheorem{Theorem}{Theorem}[section]
\newtheorem{Lemma}{Lemma}[section]
\newtheorem{Proposition}{Proposition}[section]
\newtheorem{Remark}{Remark}[section]

\numberwithin{equation}{section}
\usepackage[left=1 in, right=1 in,top=1 in, bottom=1 in]{geometry}

\def\XXint#1#2#3{{\setbox0=\hbox{$#1{#2#3}{\int}$ }
\vcenter{\hbox{$#2#3$ }}\kern-.6\wd0}}

\def\r3{\mathbb{R}^3}

\makeatletter
\@namedef{subjclassname@2020}{\textup{2020} Mathematics Subject Classification}
\makeatother

\allowdisplaybreaks
\begin{document}
\bibliographystyle{plain}

\title[I\MakeLowercase{ncompressible viscoelastic flows}]{D\MakeLowercase{iffusion} \MakeLowercase{wave phenomena and optimal time decay for incompressible viscoelastic flows}}

\author{Shenghan Li}

\address{School of Mathematical Sciences, South China Normal University, Guangzhou, Guangdong 510631, China}
\email[Shenghan Li]{shenghanliscnu@163.com}

\author{Yong Wang}

\address{School of Mathematical Sciences, South China Normal University, Guangzhou, Guangdong 510631, China}
\email[Yong Wang]{wangyongxmu@163.com}

\thanks{Corresponding author: Yong Wang, wangyongxmu@163.com}

\begin{abstract}
Motivated by the work of D. Hoff and K. Zumbrun (Indiana Univ. Math. J. 44: 603-676, 1995), we investigate the diffusion wave phenomena in three-dimensional incompressible viscoelastic flows. By employing the representation formula of the wave equation and the stationary phase methods on the sphere $\mathbb{S}^{d-1}$, we establish $L^p$ decay estimates for the solution over the whole range $1\leq p \leq \infty$, which reveals the hyperbolic nature of the incompressible viscoelastic flows.
\end{abstract}


\maketitle

\section{Introduction}
%%%%%%%%%%%%%%%%%%%%%%%%%%%%%%%%%%%%%%%%%%%%%%
Viscoelastic fluids are a unique class of materials that exhibit both viscous (liquid-like) and elastic (solid-like) properties simultaneously. This dual behavior sets them apart from simple Newtonian fluids (like water) and purely elastic solids (like a rubber band). Common examples of viscoelastic fluids in daily life, such as polymer solutions, melts, molten plastics, shampoo, hair gel, liquid soaps, biological fluids like blood, mucus, synovial fluid paints and printer inks, have been widely studied in the literature (e.g., \cite{Larson1995}).

In this article, we study the incompressible viscoelastic flows: $(t, x) \in(0,+\infty) \times \mathbb{R}^3$
\begin{align}\label{incompressible viscoelastic flow-202512201217}
\left\{
\begin{array}{lll}
\partial_t u+u \cdot \nabla u+\nabla P=\mu \Delta u+\operatorname{div}(\mathbb{F} \mathbb{F}^{\top}), \\
\partial_t \mathbb{F}+u \cdot \nabla \mathbb{F}=(\nabla u) \mathbb{F}, \\
\operatorname{div}u=0.
\end{array}
\right.
\end{align}
Here, the vector $u(t, x):[0,+\infty) \times \mathbb{R}^3 \rightarrow \mathbb{R}^3$ denotes the velocity field of materials, the scalar function $P(t, x):[0,+\infty) \times \mathbb{R}^3 \rightarrow \mathbb{R}$ denotes the pressure, the matrix $\mathbb{F}(t, x):[0,+\infty) \times \mathbb{R}^3 \rightarrow \mathbb{R}^{3 \times 3}$ stands for the deformation tensor, $\mu>0$ is the viscosity.
We supplement \eqref{incompressible viscoelastic flow-202512201217} with the initial condition
\begin{align}\label{initial-condition-202512201220}
(u, \mathbb{F})(x, t)|_{t=0}=(u_0, \mathbb{F}_0)(x) \rightarrow(\bar{\rho}, 0, \mathbb{I}),\quad |x| \rightarrow \infty.
\end{align}
We enforce the following constraints:
\begin{align}\label{important properties1-202512201228}
\operatorname{div} \mathbb{F}_0^T=0,
\end{align}
\begin{align}\label{important properties2-202512201228}
\operatorname{det} \mathbb{F}_0=1,
\end{align}
\begin{align}\label{important properties3-202512201228}
\mathbb{F}_0^{l k} \nabla_l \mathbb{F}_0^{i j}=\mathbb{F}_0^{l j} \nabla_l \mathbb{F}_0^{i k} .
\end{align}
As established in \cite{Lei-Liu-Zhou2008,Lin-Liu-Zhang2005,Liu-Walkington2001,Temam1979}, the conditions \eqref{important properties1-202512201228}, \eqref{important properties2-202512201228} and \eqref{important properties3-202512201228} are preserved for all $t \geq 0$:
\begin{align}\label{important properties1-202512201235-t}
\operatorname{div} \mathbb{F}^T=0,
\end{align}
\begin{align}\label{important properties2-202512201235-t}
\operatorname{det} \mathbb{F}=1,
\end{align}
\begin{align}\label{important properties3-202512201235-t}
\mathbb{F}^{l k} \nabla_l \mathbb{F}^{i j}=\mathbb{F}^{l j} \nabla_l \mathbb{F}^{i k} .
\end{align}

This paper investigates the diffusion wave phenomena and the large-time behavior of solutions to \eqref{incompressible viscoelastic flow-202512201217}-\eqref{important properties3-202512201228} near the motionless state $(u,\mathbb{F})=(0,\mathbb{I})$,
where $\mathbb{I}$ is the $3 \times 3$ identity matrix, for convenience in the following discussion, we define $S=(u-0,\mathbb{F}-\mathbb{I})\triangleq(u,\mathbb{E})$.

System \eqref{incompressible viscoelastic flow-202512201217} describes the flow of a viscous, incompressible fluid coupled with an elastic material, where the elastic response follows the Hookean linear elasticity $W(\mathbb{F}) = |\mathbb{F}|^2$. The equations consist of the incompressible Navier-Stokes equations coupled to a first-order hyperbolic system for the deformation tensor $\mathbb{F}$. This structure places it within the category of quasilinear parabolic-hyperbolic systems. Further physical details can be found in \cite{Gurtin1981,Lei-Liu-Zhou2008,Lin2012,Lin-Liu-Zhang2005,Sideris-Thomases2005}.

If one ignores the contribution of the deformation tensor $\mathbb{F}$, then \eqref{incompressible viscoelastic flow-202512201217} reduces to the classical incompressible Navier-Stokes equations.
The large-time behavior of solutions to the incompressible Navier-Stokes equations is a classical topic, with important contributions including \cite{Brandolese2004,Fujigaki-Miyakawa2001,Kajikiya-Miyakawa1986,Kato1984,Schonbek1985,Schonbek1986,Schonbek1991,Schonbek1992}. Leray \cite{Leray1934} constructed weak solutions for the incompressible Navier-Stokes equations in $R^3$. Kajikiya and Miyakawa \cite{Kajikiya-Miyakawa1986} derived the decay estimate provided that the initial perturbation $u_0 \in L_\sigma^2\left(\mathbb{R}^n\right) \cap L_\sigma^r \left(\mathbb{R}^n\right)(1 \leqslant r<2)$:
\begin{align}\label{INS-the decay estimate-202512201322}
\|u(t)\|_{L^2\left(\mathbb{R}^n\right)} \leqslant C(1+t)^{-\frac{n}{2}\left(\frac{1}{r}-\frac{1}{2}\right)}, \quad \forall t \geqslant 0.
\end{align}
We refer to \cite{Chen-Pan-Tong2021,Hoff-Zumbrun1995,Matsumura-Nishida1979,Liu-Wang1998,Ponce1985,Wang-Wen2022} for the large time behavior of the solutions for the compressible Navier-Stokes equations.

We next review related works that consider the deformation tensor $\mathbb{F}$, namely, those dealing with incompressible viscoelastic systems. Lin, Liu and Zhang \cite{Lin-Liu-Zhang2005} established local existence and global existence (with small initial data) of classical solutions for the initial value problem \eqref{incompressible viscoelastic flow-202512201217}-\eqref{important properties3-202512201228} see also \cite{Chen-Zhang2006,Lei-Liu-Zhou2008}. There are much important progress on the viscoelastic flows,
we refer the readers to \cite{Ai-Wang2025,Bai-Zhang2023,Chemin-Masmoudi2001,Chen-Wu2018,Gu-Wang-Xie2024,Han-Zi2020,Hu-Wang2010,
Hu-Wang2011,Cai-Lei-Lin-Masmoudi2019,Hu-Lin2016,Hu-Wang2012,Hu-Wang2015,Hu-Zhao2020JDE,Hu-Zhao2020ARMA,Hu-Wu2013,Ishigaki2020,Ishigaki2022,Jiang-Jiang2021,
Jiang-Jiang-Wu2017,Jiang-Wu-Zhong2016,
Lei2010,Lei-Liu-Zhou2007,Lei-Zhou2005,Li-Wei-Yao2016JMP,Lin-Zhang2008,Qian2011,Qian-Zhang2010,Shen-Wang-Wu2025,Tan-Wang-Wu2020, Wang-Shen-Wu-Zhang2022,Wang-Wu2021,Wang-Yang-Yuan2025,Wu-Gao-Tan2017,Wu-Wang2023,Zhang-Fang2012} and the references cited therein. Hu and Wu \cite{Hu-Wu2015} also showed that if the initial perturbation $(u_0, \mathbb{E}_0)$ belongs to $L^1\left(\mathbb{R}^3\right) \cap H^2\left(\mathbb{R}^3\right)$, the $L^2$ decay estimates:
\begin{align}\label{202512261520-decay estimates-Hu-Wu2015}
\left\|\nabla^k (u,\mathbb{E})(t)\right\|_{L^2\left(\mathbb{R}^3\right)} \leq C(1+t)^{-\frac{3}{4}-\frac{k}{2}},
\quad k=0,1.,
\end{align}
via the Fourier splitting method and the Hodge decomposition.

Note that the incompressible viscoelastic system is a parabolic-hyperbolic coupled system. We point out that the above $L^2$ decay estimates \eqref{202512261520-decay estimates-Hu-Wu2015} only exhibit the parabolic properties of the system \eqref{incompressible viscoelastic flow-202512201217}, but they cannot reveal the information on hyperbolic properties.

Crucially, one should observe that:
\begin{align*}
\left\|(u, \tilde{  \mathbb{E} }_c )(t)\right\|_{L^p\left(\mathbb{R}^3\right)} \sim \left\|\mathcal{F}^{-1}\left(e^{-|\xi|^2 t} \cdot e^{i\xi t}\right)\right\|_{L^p\left(\mathbb{R}^3\right)},
\quad
1 \leq p \leq \infty,
\end{align*}
where $\mathcal{F}^{-1}\left( e^{i\xi t}\right) $ is a  fundamental solution of the wave equation.

As a result, the $L^2$ decay estimates (or the pure energy method) might hide the wave information:
\begin{align*} \left\|\mathcal{F}^{-1}\left(e^{-|\xi|^2 t} \cdot e^{i\xi t}\right)\right\|_{L^2\left(\mathbb{R}^3\right)}
=\left\|\mathcal{F}^{-1}\left(e^{-|\xi|^2 t}\right)\right\|_{L^2\left(\mathbb{R}^3\right)}
\sim (1+t)^{-\frac{3}{4}}.
\end{align*}

Due to the diffusion wave phenomena  observed in \cite{Hoff-Zumbrun1995,Hoff-Zumbrun1997}, we shall establish some $L^p(1 \leq p \leq \infty, p \neq 2)$ decay estimates which reflect the hyperbolic aspect of the incompressible viscoelastic fluids \eqref{incompressible viscoelastic flow-202512201217} in the following.

Now we review the known results on the diffusion wave phenomena and the large time behavior of solution of the compressible Navier-Stokes (CNS) equations around a motionless state $(\rho, m)=(\rho,\rho u)=(1,0)$:
\begin{align}\label{NS}
\left\{
\begin{array}{lll}
\partial_t\rho+\operatorname{div} m=0, \\
\partial_t m^j+\operatorname{div}(\frac{m^j m}{\rho})+P(\rho)_{x_j}=\varepsilon \Delta(\frac{m^j}{\rho})+\eta \operatorname{div}(\frac{m}{\rho})_{x_j},\\
(\rho, m)|_{t=0}=(\rho_0, m_0).
\end{array}
\right.
\end{align}
Matsumura and Nishida \cite{Matsumura-Nishida1979,Matsumura-Nishida1980} proved global existence for the Cauchy problem of \eqref{NS}; they also obtained optimal time-decay rates under the condition that the initial perturbation is small in $H^3(\mathbb{R}^3) \cap L^1(\mathbb{R}^3)$:
\begin{align*}
\|\nabla^k(\varrho(t), m(t))\|_{L^2(\mathbb{R}^3)} \leq C(1+t)^{-\frac{3}{4}-\frac{k}{2}},\quad k=0,1,
\end{align*}
where $S(x, t)=(\varrho, m)(x, t)=(\rho-1, \rho u)(x, t)$.
Gao, Li and Yao \cite{Gao-Li-Yao2023} first show that the $N$-th order spatial derivative of global solution of the compressible Navier-Stokes equations tends to zero at the rate $(1+t)^{-\frac{3}{4}-\frac{N}{2}}$ instead of $(1+t)^{-\frac{3}{4}-\frac{N-1}{2}}$ stated in \cite{Wang2012}.
The concept of diffusion waves was introduced to one-dimensional problems by Liu \cite{Liu1985}, through his work on the stability of viscous shocks, rarefactions, and contact discontinuities.
Hoff and Zumbrun \cite{Hoff-Zumbrun1995,Hoff-Zumbrun1997} first studied the multidimensional diffusion wave in the compressible Navier-Stokes equations and obtained the following decay estimates:
\begin{align}\label{Hoff-Zumbrun1995estimates1}
\|S(t)\|_{L^p(\mathbb{R}^n)} \leq \begin{cases}C(1+t)^{-\frac{n}{2}(1-\frac{1}{p})-\frac{n-1}{4}(1-\frac{2}{p})} L(t), & 1 \leq p<2, \\ C(1+t)^{-\frac{n}{2}(1-\frac{1}{p})}, & 2 \leq p \leq \infty,\end{cases}
\end{align}
\begin{align}\label{Hoff-Zumbrun1995estimates2}
\|S(t)-G(t) * S_0\|_{L^p(\mathbb{R}^n)} \leq \begin{cases}C(1+t)^{-\frac{n}{2}(1-\frac{1}{p})-\frac{n}{4}(1-\frac{2}{p})-\frac{1}{2}+\sigma} L(t), & 1 \leq p<2, \\ C(1+t)^{-\frac{n}{2}(1-\frac{1}{p})-\frac{1}{2}}, & 2 \leq p \leq \infty,\end{cases}
\end{align}
\begin{align}\label{Hoff-Zumbrun1995estimates3}
\left\|\left(\begin{array}{c}
\rho(t)-1 \\
M(t)-\mathcal{K}_h(t) * M_{0, i n}
\end{array}\right)\right\|_{L^p{(\mathbb{R}^n)}} \leq C(1+t)^{-\frac{n}{2}(1-\frac{1}{p})-\frac{n-1}{4}(1-\frac{2}{p})},
\quad 2 \leq p \leq \infty,
\end{align}
where $S(t)=(\varrho(t), m(t))=(\rho(t)-1, \rho(t) u(t))$, $L(t)=\log (1+t)$ when $n=2$, and $L(t)=1$ when $n \geq 3$, $\sigma$ is any positive, $G=G(t)$ is the Green function of linearized compressible Navier-Stokes equations, $\mathcal{K}_h=\mathcal{K}_h(t, x)=\mathcal{F}^{-1}(e^{-h|\xi|^2 t})$ denotes the standard heat kernel, $M_{0, \text { in }}=\mathcal{F}^{-1}\{(\mathbb{I}-\frac{\xi \xi^{\top}}{|\xi|^2}) M_0\} $ denotes a divergence-free part of $M_0$. Estimates \eqref{Hoff-Zumbrun1995estimates1} and \eqref{Hoff-Zumbrun1995estimates2} indicated that the nonlinear terms in the Duhamel formula decay faster than the linear ones and are therefore negligible for large time. Consequently, the solution to the Cauchy problem for \eqref{NS} is asymptotically approximated by the solution of its linearized counterpart. Thus, the asymptotic profile of the solution in $L^p(\mathbb{R}^n)$ ($1 \leq p \leq \infty$, $n \geq 2$) is given by
\begin{align}\label{Hoff-Zumbrun1995asymptotic1}
S(t) \sim \underbrace{G(t)*S_0}_{\text {solutions to linearized CNS }}+\underbrace{\cdots}_{\text {nonlinear parts }} \text { in } L^p(\mathbb{R}^n),
\quad
 t \rightarrow \infty;
\end{align}
\begin{align}\label{Hoff-Zumbrun1995asymptotic3}
S(t) \sim \underbrace{G(t)*S_0}_{\text {solutions to linearized CNS }} \text { in } L^p(\mathbb{R}^n),
\quad
 t \rightarrow \infty.
\end{align}
Moreover, according to the results in \cite{Hoff-Zumbrun1995}:
\begin{align}\label{Hoff-Zumbrun1995asymptotic2}
S(t)
\sim \underbrace{G(t)*S_0}_{\text {solutions to linearized CNS}}
\sim \underbrace{\left(\begin{array}{c}
0 \\
\mathcal{K}_h(t) * M_{0, \text { in }}
\end{array}\right)}_{\text {the incompressible part}}+\underbrace{\left(\begin{array}{c}
\rho(t)-1 \\
m(t)-\mathcal{K}_h(t) * m_{0, \text { in }}
\end{array}\right)}_{\text {the diffusion wave part}} \text { in } L^p(\mathbb{R}^n),
\quad
 t \rightarrow \infty.
\end{align}
Estimate \eqref{Hoff-Zumbrun1995asymptotic2} shows that the solution to the linearized compressible Navier-Stokes equations is asymptotically the sum of two components: an incompressible part, decaying at the heat kernel rate in all $L^p(\mathbb{R}^n)$, and a diffusion wave, given by convolving the heat kernel with the fundamental solution of the wave equation. The diffusion wave decays as the heat kernel rate multiplied by $(1+t)^{-\frac{n-1}{4}(1-\frac{2}{p})}$. Consequently, for $p>2$, it decays faster than the heat kernel, making the incompressible part dominant in $L^p(\mathbb{R}^n)$. Conversely, for $1 \leq p < 2$, the diffusion wave decays more slowly and thus dominates. In this regime $1 \leq p < 2$, the perturbation $S(x,t)$ can even grow in certain norms due to diffusion wave phenomena, for instance, $\|S(\cdot, t)\|_{L^1(\mathbb{R}^3)} \sim t^{\frac{1}{2}}$.

Given the results in \cite{Hoff-Zumbrun1995}, it is natural to conjecture that system \eqref{incompressible viscoelastic flow-202512201217} also exhibits diffusion-wave phenomena, now influenced by both sound waves and elastic shear waves. To examine this, we consider its linearization about the equilibrium $(0, \mathbb{I})$:
\begin{align}\label{linearized-system-around-202512201352}
\partial_t S+L S=0.
\end{align}
Let $S=(u-0,\mathbb{F}-\mathbb{I})\triangleq(u,\mathbb{E})$ and $L$ be the linearized operator, given by
\begin{align}\label{linearized operator-202512201356}
L=\left(\begin{array}{ccc}
-\mu \Delta & -\Delta^{-1}  \operatorname{div}\operatorname{curl}\operatorname{div} \\
-\nabla & 0
\end{array}\right).
\end{align}
We then find that the velocity $u=u_s \triangleq \mathcal{F}^{-1}(\hat{\mathcal{P}}(\xi) \hat{u})$ obeys the following linear
parabolic-hyperbolic system:
\begin{align}\label{linear symmetric parabolic-hyperbolic system-202512201402}
\left\{\begin{array}{l}
\partial_t u-\mu \Delta u- \Delta^{-1}  \operatorname{div}\operatorname{curl}\operatorname{div}  \mathbb{E}=0, \\
\partial_t \mathbb{E}- \nabla u=0,
\end{array}\right.
\end{align}
where $\hat{\mathcal{P}}(\xi)=\mathbb{I}-\frac{\xi \xi^{\top} }{|\xi|^2}, \xi \in \mathbb{R}^3$, equivalently, the velocity $u$ satisfies the following strongly damped wave equation:
\begin{align}\label{strongly damped wave equation-202512201409}
\partial_t^2 u- \Delta u-\mu \partial_t \Delta u=0 .
\end{align}

An explicit expression for $u(x, t)$ can be given by the formula:
\begin{align}\label{d-Fourier inverse transform}
u(x, t)
\nonumber
&
=\mathcal{F}^{-1}\{ \hat{u}(\xi, t)\}
=\mathcal{F}^{-1}\{\frac{e^{\lambda_1 t}-e^{\lambda_2 t}}{\lambda_1-\lambda_2}\left[\left(\mathbb{I}-\frac{\xi \xi^{\top}}{|\xi|^2}\right) i \hat{\mathbb{E}}_0 \xi\right]
+\frac{\lambda_1 e^{\lambda_1 t}-\lambda_2 e^{\lambda_2 t}}{\lambda_1-\lambda_2} \hat{u}_0 \}\\
\nonumber
&
=\mathcal{F}^{-1}\{e^{-\frac{\mu}{2}|\xi|^2 t}\frac{\sin (b t)}{b} \left[\left(\mathbb{I}-\frac{\xi \xi^{\top}}{|\xi|^2}\right) i \hat{\mathbb{E}}_0 \xi\right]
+e^{-\frac{\mu}{2}|\xi|^2 t}\left[\cos (b t)-\frac{\mu}{2}|\xi|^2 \frac{\sin (b t)}{b}\right] \hat{u}_0 \}\\
&
\sim
\mathcal{F}^{-1}\{e^{-\frac{\mu}{2}|\xi|^2 t}\hat{\omega}(\xi, t) \left[\left(\mathbb{I}-\frac{\xi \xi^{\top}}{|\xi|^2}\right) i \hat{\mathbb{E}}_0 \xi\right]
+e^{-\frac{\mu}{2}|\xi|^2 t}\left[\hat{\omega}_{ t}(\xi, t)-\frac{\mu}{2}|\xi|^2 \hat{\omega}(\xi, t) \right] \hat{u}_0 \},
 \quad|\xi|  \ll 1,
\end{align}
where $\omega$ is the fundamental solution to the wave equation in $\mathbb{R}^n$, i.e.
\begin{align}
\left\{
\begin{array}{lll}
\partial_{t t}\omega-\beta^2 \Delta \omega  =0, \\ \nonumber
\omega(0)  =0, \\
\omega_t(0)  =\delta,
\end{array}
\right.
\end{align}
where $\delta$ is the Dirac distribution, and that the Fourier transform of $\omega$ is given by
\begin{align*}
\hat{\omega}(\xi, t)=(2 \pi)^{-n / 2} \frac{\sin (\beta|\xi| t)}{\beta|\xi|}, \quad
\hat{\omega}_{ t}(\xi, t)=(2 \pi)^{-n / 2} \cos (\beta|\xi| t).
\end{align*}
Considering the eigenvalues above, we observe that the velocity exhibits distinct behavior in low- and high-frequency regimes: the low-frequency part decays algebraically, whereas the high-frequency part decays exponentially.
For low frequencies, the eigenvalues are
\begin{align}\label{divergence-freeeigenvalues34-202512261906-1}
\lambda_{1,2}=-\frac{\mu}{2}|\xi|^2 \pm i |\xi| \frac{\sqrt{4-\mu^2|\xi|^2}}{2}
\sim  -\frac{\mu}{2}|\xi|^2 \pm i  |\xi| ,\quad|\xi|  \ll 1,
\end{align}
where $\lambda_{1,2}$ are the roots of the equation:
\begin{align}\label{rootsequation34-202512261907-2}
\lambda^2+\mu|\xi|^2 \lambda+|\xi|^2=0.
\end{align}
\eqref{divergence-freeeigenvalues34-202512261906-1} suggests that diffusion wave phenomena can be expected in the low-frequency regime of the velocity, motivated by the observations in \cite{Hoff-Zumbrun1995}.

In contrast to the settings in \cite{Hoff-Zumbrun1995, Kobayashi-Shibata2002}, where the deformation tensor $\mathbb{F}$ is neglected, the solution of the linearized system \eqref{linearized-system-around-202512201352} exhibits different behavior. A key difference stems from the additional hyperbolic aspect arising from the elastic shear wave.

This paper establishes that for an initial perturbation $S_0=\left(u_0, \mathbb{F}_0-\mathbb{I}\right)=\left(u_0, \mathbb{E}_0\right)$ small enough in $L^1 \cap H^3$, the global strong solution obeys the following decay estimates:
\begin{align}\label{DW-incompressible viscoelastic flow-infty-202512201344}
\| (u, \tilde{  \mathbb{E} }_c )(t)\|_{L^p} \leq C(p)(1+t)^{-\frac{3}{2}\left(1-\frac{1}{p}\right)
-\frac{1}{2}\left(1-\frac{2}{p}\right)},
\quad
 1 \leq p \leq \infty,
\end{align}
\begin{align}\label{L1-grow-202512271441}
\|(u,\tilde{  \mathbb{E} }_c)(t)\|_{L^1(\mathbb{R}^3)} \sim (1+t)^{\frac{1}{2}},
\quad
 t \geq 0.
\end{align}
Here $\tilde{  \mathbb{E} }_c\triangleq \mathcal{F}^{-1}(\hat{\mathcal{Q}}(\xi) \hat{ \mathbb{E} })$,
$\hat{\mathcal{Q}}(\xi)=\frac{\xi \xi^{\top} }{|\xi|^2}, \xi \in \mathbb{R}^3$.
This result \eqref{DW-incompressible viscoelastic flow-infty-202512201344}- \eqref{L1-grow-202512271441} reveals the hyperbolic nature of the incompressible viscoelastic fluids and reflects the influence of diffusion wave phenomena on the decay rate of the solution.
Furthermore, this result \eqref{DW-incompressible viscoelastic flow-infty-202512201344} improves the decay rate of the $L^p$ norm of the perturbation $S$ obtained in \cite{Hu-Wu2015} for $p>2$.

The proof of the main result is outlined as follows. The decay of $\|S(t)\|_{L^p}$ for $(1\leq p \leq \infty)$, where $S=(u, \mathbb{E})$, is derived from the integral equation
\begin{align*}
S(t)=e^{-t L} S(0)+\int_0^t e^{-(t-\tau) L} N(S) \mathrm{d} \tau,
\end{align*}
where $N(S)=\left(N_1(S), N_2(S) \right)$ is a nonlinearity.
We decompose $S$ into its low-frequency part $S_1$ and high-frequency part $S_{\infty}$. For the low-frequency component $S_1$, linearized analysis yields the decay estimates in the range $2 \leq p \leq \infty$. The high-frequency part $S_{\infty}$ is controlled via a Poincar¨¦-type inequality in high-frequency space combined with $L^2$ decay estimates for derivatives of all orders. For $1\leq p<2$, the $L^p$ estimate of $S(t)$ follows from the results in \cite{Shibata2000, Kobayashi-Shibata2002}. Moreover, using the pure energy method developed in \cite{Guo-Wang2012, Gao-Li-Yao2023}, we establish the optimal $L^2$ decay rate for higher-order spatial derivatives of the global solution to the incompressible viscoelastic fluids.

This paper is organized as follows. In Section \ref{202512201502-Theorem-proof}, we prove Theorem \ref{diffusion waves-infty-202512201451}. Section \ref{202512201720-Theorem-proof} is devoted to the proof of Theorem \ref{diffusion waves-1-202512201453}. In Sections \ref{Energy Estimates-202512221729} and \ref{202512231847-Negative Sobolev/Besov estimates-202512231853}, we establish the delicate energy estimates and the negative Sobolev/Besov estimates for the solutions, respectively, which constitute the core of the pure energy method. Using these estimates, we complete the proof of Theorem \ref{pure-energy-202512221702}\eqref{202512221715-globalsolution-incompressible viscoelastic flows}-\eqref{pure-energy-202512221703} in Section \ref{202512241526-}. Section \ref{202512241814-Theorem-N-pure-energy-202512221703} then provides the proof of Theorem \ref{pure-energy-202512221702}\eqref{N-pure-energy-202512221703}. Finally, the Appendix \ref{appendix} collects several useful lemmas that are frequently employed in the preceding sections.

\noindent{\bf Notations.}
In this section, we prepare notations and function spaces which will be used throughout the paper. $L^p(1 \leq p \leq \infty)$ denotes the usual Lebesgue space on $\mathbb{R}^3$, and its norm is denoted by $\|\cdot\|_{L^p}$. Similarly $W^{m, p}(1 \leq p \leq \infty, m \in\{0\} \cup \mathbb{N})$ denotes the $m$-th order $L^p$ Sobolev space on $\mathbb{R}^3$, and its norm is denoted by $\|\cdot\|_{W^{m, p}}$. We define $H^m=W^{m, 2}$ for an integer $m \geq 0$.
The inner product of $L^2$ is denoted by
\begin{align*}
(f, g):=\int_{\mathbb{R}^3} f(x) \overline{g(x)} d x, f, g \in L^2 .
\end{align*}
For functions $f=f(x)$ and $g=g(x)$, we denote the convolution of $f$ and $g$ by $f * g$ :
\begin{align*}
(f * g)(x)=\int_{\mathbb{R}^3} f(x-y) g(y) \mathrm{d} y .
\end{align*}
We denote the Fourier transform of a function $f=f(x)$ by $\hat{f}$ or $\mathcal{F} f$ :
\begin{align*}
\hat{f}(\xi)=(\mathcal{F} f)(\xi)=\frac{1}{(2 \pi)^{\frac{3}{2}}} \int_{\mathbb{R}^3} f(x) e^{-i \xi \cdot x} \mathrm{~d} x
\quad
(\xi \in \mathbb{R}^3).
\end{align*}
The Fourier inverse transform is denoted by $\mathcal{F}^{-1}$ :
\begin{align*}
(\mathcal{F}^{-1} f)(x)=\frac{1}{(2 \pi)^{\frac{3}{2}}} \int_{\mathbb{R}^3} f(\xi) e^{i \xi \cdot x} \mathrm{~d} \xi
\quad
(x \in \mathbb{R}^3).
\end{align*}
We use $C$ for a generic positive constant, and denote $A \leq C B$ by $A \lesssim B$. Furthermore, we say $A \approx B$ if $A \lesssim B$ and $B \lesssim A$.

We introduce the strain tensor defined by $\mathbb{E}=\mathbb{F}-\mathbb{I}$,
then the original system \eqref{incompressible viscoelastic flow-202512201217} can be re-written in terms of $(u, \mathbb{E}) : (t, x) \in(0,+\infty) \times \mathbb{R}^n$
\begin{align}\label{incompressible viscoelastic flow-E-202512201224}
\left\{
\begin{array}{lll}
\partial_t u-\mu \Delta u+\nabla P-\operatorname{div}\mathbb{E}=\operatorname{div}(\mathbb{E} \mathbb{E}^{\top})-u \cdot \nabla u, \\
\partial_t \mathbb{E}-\nabla {u}=(\nabla u) \mathbb{E}-u \cdot \nabla \mathbb{E}, \\
\operatorname{div}u=0.
\end{array}
\right.
\end{align}
We supplement \eqref{incompressible viscoelastic flow-E-202512201224} with the initial condition
\begin{align}\label{initial-condition-202512201225}
(u, \mathbb{E})(x, t)|_{t=0}=(u_0, \mathbb{E}_0)(x) \rightarrow(0, 0),
\quad
|x| \rightarrow \infty.
\end{align}
Here $g_j, j=1,2$, denote the nonlinear terms;
\begin{align*}
& g_1=\operatorname{div}(\mathbb{E} \mathbb{E}^{\top})-u \cdot \nabla u, \\
& g_2=(\nabla u) \mathbb{E}-u \cdot \nabla \mathbb{E}.
\end{align*}
Applying the operator $\Delta^{-1}  \operatorname{div}\operatorname{curl}$ to \eqref{incompressible viscoelastic flow-E-202512201224}-1,
then the original system \eqref{incompressible viscoelastic flow-E-202512201224} can be re-written in terms of $(u, \mathbb{E}) : (t, x) \in(0,+\infty) \times \mathbb{R}^n$
\begin{align}\label{incompressible viscoelastic flow-E-202512201435}
\left\{
\begin{array}{lll}
\partial_t u-\mu \Delta u-\Delta^{-1}  \operatorname{div}\operatorname{curl}\operatorname{div}\mathbb{E}=N_1, \\
\partial_t \mathbb{E}-\nabla {u}= N_2, \\
\operatorname{div}u=0.
\end{array}
\right.
\end{align}
We supplement \eqref{incompressible viscoelastic flow-E-202512201435} with the initial condition
\begin{align}\label{initial-condition-202512201435}
(u, \mathbb{E})(x, t)|_{t=0}=(u_0, \mathbb{E}_0)(x) \rightarrow(0, 0),
\quad
|x| \rightarrow \infty.
\end{align}
Here $N_j, j=1,2$, denote the nonlinear terms;
\begin{align*}
& N_1=\Delta^{-1}  \operatorname{div}\operatorname{curl} g_1, \\
& N_2=(\nabla u) \mathbb{E}-u \cdot \nabla \mathbb{E}.
\end{align*}
We first recall the $L^2$ decay estimates obtained in \cite{Hu-Wu2015}.
\begin{Proposition}\label{2015Hu-Wu-202512201444}(\cite{Hu-Wu2015})
Suppose that $n=3$ and the initial data $u_0, \mathbb{E}_0 \in L^1\left(\mathbb{R}^3\right) \cap H^k\left(\mathbb{R}^3\right)$ ( $k \geq 2$ being an integer) fulfill the assumptions \eqref{important properties1-202512201228}-\eqref{important properties3-202512201228}. If the initial data satisfy $\left\|u_0\right\|_{H^2}+\left\|\mathbb{E}_0\right\|_{H^2} \leq \delta$ for certain sufficiently small $\delta>0$, then the Cauchy problem \eqref{incompressible viscoelastic flow-E-202512201224}-\eqref{initial-condition-202512201225} admits a unique global classical solution $(u, \mathbb{E}) $ such that
\begin{align*}
\left\{\begin{array}{l}
\partial_t^j \nabla^\alpha u \in L^{\infty}\left(0, T ; H^{k-2 j-|\alpha|}\left(\mathbb{R}^3\right)\right) \cap L^2\left(0, T ; H^{k-2 j-|\alpha|+1}\left(\mathbb{R}^3\right)\right), \\
\partial_t^j \nabla^\alpha \mathbb{E} \in L^{\infty}\left(0, T ; H^{k-2 j-|\alpha|}\left(\mathbb{R}^3\right)\right),
\end{array}\right.
\end{align*}
for all integer $j$ and multi-index $\alpha$ satisfying $2 j+|\alpha| \leq k$.
For all $t \geq 0$, the following decay estimates hold
\begin{align}\label{2015Hu-Wu-Decay}
\|u(t)\|_{L^2}+\|\mathbb{E}(t)\|_{L^2} & \leq C M(1+t)^{-\frac{3}{4}}, \\
\|\nabla u(t)\|_{H^1}+\|\nabla \mathbb{E}(t)\|_{H^1} & \leq C M(1+t)^{-\frac{5}{4}},
\end{align}
where $M=\left\|u_0\right\|_{L^1 \cap H^2}+\left\|\mathbb{E}_0\right\|_{L^1 \cap H^2}$.

Moreover, if the Fourier transforms of the initial data $\left(u_0, n_0\right)$ (where $n_0= \Lambda^{-1} \operatorname{div} \mathbb{E}_0$, and the operator $\Lambda$ is defined in \cite{Hu-Wu2015} below) also satisfy $\left|\widehat{u_{i 0}}\right| \geq c_0$, $\left|\widehat{n_{i 0}}\right| \geq c_0$ for $0 \leq|\xi| \ll 1$, where the lower bound $c_0>0$ satisfies $c_0 \sim O\left(\delta^\zeta\right)$ with $\zeta \in(0,1)$, then there exists a $t_0 \gg 1$ such that
\begin{align}\label{L2lower bound-202512201744}
\|u(t)\|_{L^2}+\|\mathbb{E}(t)\|_{L^2} \geq C(1+t)^{-\frac{3}{4}}, \quad \forall t \geq t_0.
\end{align}
\end{Proposition}
We now state the three main results of this paper, which reflect the effect of diffusion wave phenomena on decay properties.
\begin{Theorem}\label{diffusion waves-infty-202512201451}
For $2 \leq p \leq \infty$, there exists a positive number $\epsilon$ with the following property: if the initial data $S_0=\left(u_0, \mathbb{E}_0\right)$ belongs to $H^3 \cap  L^1$ and satisfies $\left\|S_0\right\|_{H^3} \leq \epsilon$, then the Cauchy problem \eqref{incompressible viscoelastic flow-E-202512201224}-\eqref{initial-condition-202512201225} possesses a unique global solution
$S(t)=(u(t),\mathbb{E}(t)) \in C\left([0, \infty) ; H^3\right)$ such that
\begin{align}\label{DW-incompressible viscoelastic flow-infty}
\|u(t)\|_{L^p} \leq C(p)(1+t)^{-\frac{3}{2}\left(1-\frac{1}{p}\right)
-\frac{1}{2}\left(1-\frac{2}{p}\right)},
\end{align}
\begin{align}\label{EC-DW-incompressible viscoelastic flow-infty-202512251739}
\|\tilde{  \mathbb{E} }_c(t)\|_{L^p} \leq C(p)(1+t)^{-\frac{3}{2}\left(1-\frac{1}{p}\right)
-\frac{1}{2}\left(1-\frac{2}{p}\right)},
\end{align}
uniformly for $t \geq 0$. Here $C(p)$ is a positive constant depending only on $p$, $\tilde{  \mathbb{E} }_c= \mathcal{F}^{-1}(\hat{\mathcal{Q}}(\xi) \hat{ \mathbb{E} })$, $\hat{\mathcal{Q}}(\xi)=\frac{\xi \xi^{\top} }{|\xi|^2}, \xi \in \mathbb{R}^3$.
\end{Theorem}

\begin{Remark}
Theorem \ref{diffusion waves-infty-202512201451} implies that for $2<p \leq \infty$, the $L^p$ norm of $(u,\tilde{  \mathbb{E} }_c)$ decays to zero faster than the heat kernel as $t \rightarrow \infty$, because $\frac{1}{2}\left(1-\frac{2}{p}\right)>0$ in this range. This improves the earlier result obtained in \cite{Hu-Wu2015}. Moreover, we find that the presence of the elastic force $ \operatorname{div}\left(\mathbb{F} \mathbb{F}^{\top} \right)$ leads to a faster decay rate of the $L^p$ norm for $2<p \leq \infty$, compared to the corresponding results for the incompressible Navier-Stokes equations in \cite{Kajikiya-Miyakawa1986, Schonbek1985}.
\end{Remark}

\begin{Theorem}\label{diffusion waves-1-202512201453}
For $1 \leq p <  2$, there exists a positive number $\epsilon$ with the following property: if the initial data $S_0=\left(u_0, \mathbb{E}_0\right)$ belongs to $H^3 \cap  L^1$ and satisfies $\left\|S_0\right\|_{H^3} \leq \epsilon$, then the Cauchy problem \eqref{incompressible viscoelastic flow-E-202512201224}-\eqref{initial-condition-202512201225} possesses a unique global solution
$S(t)=(u(t),\mathbb{E}(t)) \in C\left([0, \infty) ; H^3\right)$ such that
\begin{align}\label{DW-incompressible viscoelastic flow-1}
\|u(t)\|_{L^p} \leq C(p)(1+t)^{-\frac{3}{2}\left(1-\frac{1}{p}\right)
+\frac{1}{2}\left(\frac{2}{p}-1\right)},
\end{align}
\begin{align}\label{EC-DW-incompressible viscoelastic flow-1-202512251741}
\|\tilde{  \mathbb{E} }_c(t)\|_{L^p} \leq C(p)(1+t)^{-\frac{3}{2}\left(1-\frac{1}{p}\right)
+\frac{1}{2}\left(\frac{2}{p}-1\right)},
\end{align}
uniformly for $t \geq 0$. Here $C(p)$ is a positive constant depending only on $p$, $\tilde{  \mathbb{E} }_c= \mathcal{F}^{-1}(\hat{\mathcal{Q}}(\xi) \hat{ \mathbb{E} })$, $\hat{\mathcal{Q}}(\xi)=\frac{\xi \xi^{\top} }{|\xi|^2}, \xi \in \mathbb{R}^3$.
\end{Theorem}

\begin{Remark}
Under the influence of diffusion wave phenomena, Theorem \ref{diffusion waves-1-202512201453} implies that for $1 \leq p <  2$, the $L^p$ norm of $(u,\tilde{  \mathbb{E} }_c)$ decays to zero slower than the heat kernel as $t \rightarrow \infty$, because $\frac{1}{2}\left(\frac{2}{p}-1\right)>0$ in this range, which is consistent with the observation by Hoff and Zumbrun \cite{Hoff-Zumbrun1995,Hoff-Zumbrun1997}.
\end{Remark}

\begin{Theorem}\label{diffusion waves-1-sharp-202512201454}
Under the assumptions in Theorem \ref{diffusion waves-1-202512201453} and Proposition \ref{2015Hu-Wu-202512201444}, then we have
\begin{align}\label{DW-incompressible viscoelastic flow-1-202512221544}
\|u(t)\|_{L^1} \geq C(1+t)^{
+\frac{1}{2}},
\end{align}
\begin{align}\label{EC-DW-incompressible viscoelastic flow-1-202512251747}
\|\tilde{  \mathbb{E} }_c(t)\|_{L^1} \geq C(1+t)^{
+\frac{1}{2}},
\end{align}
uniformly for $t \geq 0$. Here $C$ is a positive constant, $\tilde{  \mathbb{E} }_c= \mathcal{F}^{-1}(\hat{\mathcal{Q}}(\xi) \hat{ \mathbb{E} })$, $\hat{\mathcal{Q}}(\xi)=\frac{\xi \xi^{\top} }{|\xi|^2}, \xi \in \mathbb{R}^3$.
\end{Theorem}
\begin{Remark}
In fact, Theorems \ref{diffusion waves-1-202512201453} and \ref{diffusion waves-1-sharp-202512201454} show that, due to diffusion wave phenomena, $(u,\tilde{  \mathbb{E} }_c)$ can even grow in the $L^1$ norm, with $\|(u,\tilde{  \mathbb{E} }_c)(t)\|_{L^1(\mathbb{R}^3)} \sim (1+t)^{\frac{1}{2}}$, a behavior that was also observed by Hoff and Zumbrun \cite{Hoff-Zumbrun1995,Hoff-Zumbrun1997}.
\end{Remark}

\noindent{\bf The proof of Theorem \ref{diffusion waves-1-sharp-202512201454} }
With the lower $L^2$ estimate \eqref{L2lower bound-202512201744} established, we now introduce the $L^{\infty}$ decay estimate given in Theorem \ref{diffusion waves-infty-202512201451}:
\begin{align}\label{diffusion waves-infty-spec-202512201755}
\|u(t)\|_{L^{\infty}} \leq C(1+t)^{-2}\left(\left\|S_0\right\|_{L^1}+\left\|S_0\right\|_{H^3}\right), t \geq 0 .
\end{align}
Combining the lower $L^2$ estimate \eqref{L2lower bound-202512201744}, the decay estimate \eqref{diffusion waves-infty-spec-202512201755}, and the interpolation inequality $\|u(t)\|_{L^2} \leq\|u(t)\|_{L^1}^{\frac{1}{2}}\|u(t)\|_{L^{\infty}}^{\frac{1}{2}}$, we obtain
\begin{align*}
c(1+t)^{-\frac{3}{4}} & \leq\|u(t)\|_{L^2}
\leq\|u(t)\|_{L^{\infty}}^{\frac{1}{2}}\|u(t)\|_{L^1}^{\frac{1}{2}}
\leq C(1+t)^{-1}\|u(t)\|_{L^1}^{\frac{1}{2}}.
\end{align*}
Thus, we obtain
\begin{align*}
\|u(t)\|_{L^1} \geq c(1+t)^{\frac{1}{2}} .
\end{align*}
Having established the  lower $L^2$ estimate for $\tilde{  \mathbb{E} }_c$ via a method similar to Reference \cite{Hu-Wu2015} or \cite{Chen-Pan-Tong2021}, we apply an analogous argument: $\|\tilde{  \mathbb{E} }_c(t)\|_{L^2} \geq C(1+t)^{-\frac{3}{4}}$.
Repeating this process allows us to conclude the proof of \eqref{EC-DW-incompressible viscoelastic flow-1-202512251747}.
This complete the proof of Theorem \ref{diffusion waves-1-sharp-202512201454}.

The fourth result concerns the global existence, uniqueness, and the optimal $L^2$ decay rate of the $\ell-t h(\ell=0,1, \ldots, N-1,N)$ order spatial derivatives of the solution.
\begin{Theorem}\label{pure-energy-202512221702}
Let the initial data $(u_0,\mathbb{E}_0) \in H^N(\mathbb{R}^3)$ for some integer $N \geq 3$, and suppose it satisfies the conditions
\begin{align}\label{202512231806-initial-data-conditions}
\operatorname{div}( \mathbb{E}_0^T)=0,
\quad
\nabla_m \mathbb{E}_{0 i j}-\nabla_j \mathbb{E}_{0 i m}=\mathbb{E}_{0 l j} \nabla_l \mathbb{E}_{0 i m}-\mathbb{E}_{0 l m} \nabla_l \mathbb{E}_{0 i j}
.
\end{align}
Then there exists a sufficiently small constant $\delta_0>0$ such that, provided that
\begin{align*}
\|(u_0, \mathbb{E}_0)\|_{H^3} \leq \delta_0,
\end{align*}
then the Cauchy problem \eqref{incompressible viscoelastic flow-E-202512201224}-\eqref{initial-condition-202512201225} has a unique global solution $(u, \mathbb{E})$ satisfying for all $t \geq 0$:
\begin{align}\label{202512221715-globalsolution-incompressible viscoelastic flows}
\|(u, \mathbb{E})(t)\|_{H^N}^2+\int_0^t(\|\nabla u(\tau)\|_{H^N}^2
+\|\nabla\mathbb{E}(\tau)\|_{H^{N-1}}^2) d \tau
\lesssim\|(u_0, \mathbb{E}_0)\|_{H^N}^2 .
\end{align}
Moreover, if $(u_0, \mathbb{E}_0) \in \dot{H}^{-s}$ with $s \in[0, \frac{3}{2})$ or $(u_0, \mathbb{E}_0) \in$ $\dot{B}_{2, \infty}^{-s}$ with $s \in\left(0, \frac{3}{2}\right]$, then for all $t>0$,
\begin{align}\label{202512241739-1-5}
\|(u,\mathbb{E})(t)\|_{\dot{H}^{-s}} \leq C,
\end{align}
or
\begin{align}\label{202512241740-1-6}
\|(u,\mathbb{E})(t)\|_{\dot{B}_{2, \infty}^{-s}} \leq C .
\end{align}
Furthermore, we have the following decay results:
\begin{align}\label{pure-energy-202512221703}
\|\nabla^{\ell}(u,\mathbb{E})(t)\|_{H^{N-\ell}} \leq C(1+t)^{-\frac{\ell+s}{2}}, \quad \ell=0,1, \ldots, N-1,
\end{align}
\begin{align}\label{N-pure-energy-202512221703}
\|\nabla^{\ell}( u,\mathbb{E})(t)\|_{H^{N}} \leq C(1+t)^{-\frac{\ell+s}{2}}, \quad \ell=N.
\end{align}
\end{Theorem}

\begin{Remark}
If $(u_0, \mathbb{F}_0-\mathbb{I}) \in H^N(\mathbb{R}^3)\cap\dot{B}_{2, \infty}^{-3/2}(\mathbb{R}^3) $,
Theorem \ref{pure-energy-202512221702} \eqref{pure-energy-202512221703}, together with the Gagliardo-Nirenberg inequality
\begin{align*}
\|f\|_{L^{\infty}(\mathbb{R}^3)} \leq C\|D f\|_{L^2(\mathbb{R}^3)}^{\frac{1}{2}}\|D^2 f\|_{L^2(\mathbb{R}^3)}^{\frac{1}{2}},
\end{align*}
and the interpolation inequality
\begin{align*}
\|f\|_{L^{p}(\mathbb{R}^3)} \leq\| f \|_{L^{2}(\mathbb{R}^3)}^{\frac{2}{p}}\| f \|_{L^{\infty}(\mathbb{R}^3)}^{1-\frac{2}{p}},
\quad
(2 \leq p \leq \infty),
\end{align*}
yields the decay estimate
\begin{align}\label{Tan-Wang-Wu2020decay results-Linfty-202512281533}
\|(u, \mathbb{E})(t)\|_{L^{p}(\mathbb{R}^3)} \leq C(1+t)^{-\frac{3}{2}\left(1-\frac{1}{p}\right) }.
\end{align}
Because $L^1(\mathbb{R}^3) \subset \dot{B}_{2, \infty}^{-3/2}(\mathbb{R}^3) $, the same decay result \eqref{Tan-Wang-Wu2020decay results-Linfty-202512281533} remains valid when the initial data satisfy $(u_0, \mathbb{E}_0) \in H^N(\mathbb{R}^3)\cap L^1(\mathbb{R}^3)$.
\end{Remark}
\begin{Remark}
Whereas Theorem \ref{diffusion waves-infty-202512201451} \eqref{DW-incompressible viscoelastic flow-infty}-\eqref{EC-DW-incompressible viscoelastic flow-infty-202512251739} reveals the hyperbolic aspect of the incompressible viscoelastic system \eqref{incompressible viscoelastic flow-202512201217}, the decay rate obtained in Theorem \ref{pure-energy-202512221702} \eqref{pure-energy-202512221703}-\eqref{Tan-Wang-Wu2020decay results-Linfty-202512281533} via the pure energy method (i.e., $L^2$-decay estimates) captures only its parabolic nature, even under the same type of initial data.
\end{Remark}

%\tiny         size 5pt
%\scriptsize   size 7pt
%\footnotesize size 8pt
%\small        size 9pt
%\normalsize   size 10pt
%\large        size 12pt
%\Large        size 14.4pt
%\LARGE        size 17.28pt
%\huge         size 20.74pt
%\Huge         size 24.88pt

\section{Proof of Theorem \ref{diffusion waves-infty-202512201451} }\label{202512201502-Theorem-proof}
In this section, we prove Theorem \ref{diffusion waves-infty-202512201451}. Proposition \ref{2015Hu-Wu-202512201444} and Theorem \ref{pure-energy-202512221702} already provides the global existence and the $L^2$ decay of higher-order derivatives. Consequently, the proof reduces to establishing the $L^p$ decay estimates for $p \neq 2$.
By the standard interpolation inequality: $\|u(t)\|_{L^p} \leq\|u(t)\|_{L^2}^{\frac{2}{p}}\|u(t)\|_{L^{\infty}}^{1-\frac{2}{p}}(2 \leq p \leq \infty)$, it is sufficient to establish the $L^\infty$ decay rates for $u$ and $\tilde{\mathbb{E}}_c$.

The problem \eqref{incompressible viscoelastic flow-E-202512201435}-\eqref{initial-condition-202512201435} takes the form:
\begin{align}\label{linearized-system-around-non-202512201510}
\left\{\begin{array}{l}
\partial_t S+L S=N, \\
\operatorname{div} u=0, \\
\left.S\right|_{t=0}=S_0.
\end{array}\right.
\end{align}
Here $L$ is the linearized operator given by
\begin{align}\label{linearized operator-202512201511}
L=\left(\begin{array}{ccc}
-\mu \Delta & -\Delta^{-1}  \operatorname{div}\operatorname{curl}\operatorname{div} \\
-\nabla & 0
\end{array}\right),
\quad
N=\left(\begin{array}{l}N_1 \\ N_2 \\ \end{array}\right).
\end{align}

Theorem \ref{diffusion waves-infty-202512201451} is proved by combining Proposition \ref{2015Hu-Wu-202512201444} and Theorem \ref{pure-energy-202512221702} with the following $L^\infty$ decay estimates for $u(t)$ and $\tilde{\mathbb{E}}_c(t)$.

\begin{Lemma}\label{Linfty-pure-energy-202512201519}
The following inequality holds for all $t \geq 0$, provided that there exists a positive number $\delta_0$ such that $\left\|S_0\right\|_{L^1}+\left\|S_0\right\|_{H^3} \leq \delta_0$:
\begin{align}\label{Linfty-pure-energy-decay-202512201523}
\|u(t)\|_{L^{\infty}} \leq C(1+t)^{-2}\left(\left\|S_0\right\|_{L^1}+\left\|S_0\right\|_{H^3}\right),
\end{align}
\begin{align}\label{EC-Linfty-pure-energy-decay-20251251738}
\|\tilde{  \mathbb{E} }_c(t)\|_{L^{\infty}} \leq C(1+t)^{-2}\left(\left\|S_0\right\|_{L^1}+\left\|S_0\right\|_{H^3}\right).
\end{align}
\end{Lemma}

Proving Lemma \ref{Linfty-pure-energy-202512201519} requires the following $L^2$ decay estimates for $\nabla^k u(t)$ and $\nabla^k \tilde{\mathbb{E}}_c(t)$ as prerequisites.

\begin{Lemma}\label{L2-decay-estimates-202512251734}
The following inequality holds for all $t \geq 0$, provided that there exists a positive number $\delta_0$ such that $\left\|S_0\right\|_{L^1}+\left\|S_0\right\|_{H^3} \leq \delta_0$:
\begin{align}\label{L^2-decay-estimates-202512251735}
\left\|\nabla^k u(t)\right\|_{L^2} \leq C(1+t)^{-\frac{3}{4}-\frac{k}{2}}\left(\left\|S_0\right\|_{L^1}+\left\|S_0\right\|_{H^3}\right),
\quad
 k=0,1,2,3.
\end{align}
\begin{align}\label{EC-L^2-decay-estimates-202512251735}
\left\|\nabla^k \tilde{  \mathbb{E} }_c(t)\right\|_{L^2} \leq C(1+t)^{-\frac{3}{4}-\frac{k}{2}}\left(\left\|S_0\right\|_{L^1}+\left\|S_0\right\|_{H^3}\right),
\quad
 k=0,1,2,3.
\end{align}
\end{Lemma}

\begin{proof}
Lemma \ref{L2-decay-estimates-202512251734} is a direct consequence of Theorem \ref{pure-energy-202512221702}.
\end{proof}

We next focus our analysis on the linearized problem
\begin{align}\label{linearized problem-202512201530}
\left\{\begin{array}{l}
\partial_t S+L S=0, \\
\operatorname{div} u=0, \\
\left.S\right|_{t=0}=S_0.
\end{array}\right.
\end{align}

The solution to \eqref{linearized problem-202512201530} is given by the semigroup $e^{-tL}$ (generated by $-L$) acting on the initial data: $S(t) = e^{-tL} S_0$.

Applying the Fourier transform in $x$ to $S(t)=e^{-tL}S_0$ allows us to analyze its large time behavior. We thus obtain
\begin{align}\label{linearized problem-Fourier-202512201535}
\left\{\begin{array}{l}
\partial_t \hat{S}+\hat{L}_{\xi} \hat{S}=0, \\
i \xi \cdot \hat{u}=0, \\
\left.\hat{S}\right|_{t=0}=\hat{S}_0,
\end{array}\right.
\end{align}
where
\begin{align*}
\hat{L}_{\xi} \hat{S}=\left(\begin{array}{c}
 \mu |\xi|^2 \hat{u}-i \left(\mathbb{I}-\frac{\xi \xi^{\top}}{|\xi|^2}\right) \hat{\mathbb{E}} \xi \\
-i \hat{u} \xi ^{\top}
\end{array}\right) .
\end{align*}

A direct computation yields the following expression for $e^{-t \hat{L}_{\xi}} \hat{S}_0$.
\begin{Lemma}\label{expression-202512201545}
For $|\xi| \neq 0, \frac{1}{\mu}$, the solution to \eqref{linearized problem-Fourier-202512201535} takes the form
\begin{align}\label{expression-202512201551}
&\left(\begin{array}{l}
\hat{u}(\xi, t) \\
\hat{\mathbb{E}}(\xi, t)
\end{array}\right)=\left(\begin{array}{lll}
\hat{K}^{11}(\xi, t) & \hat{K}^{12}(\xi, t) \\
\hat{K}^{21}(\xi, t) & \hat{K}^{22}(\xi, t) \\
\end{array}\right)\left(\begin{array}{c}
\hat{u}_0(\xi) \\
\hat{\mathbb{E}}_0(\xi)
\end{array}\right) .
\end{align}
Here
\begin{align*}
\hat{K}^{11}(\xi, t)=& \frac{ \lambda_1 e^{\lambda_1 t}- \lambda_2 e^{\lambda_2 t}}{\lambda_1-\lambda_2};
\end{align*}
$\hat{K}^{12}(\xi, t) \hat{\mathbb{E}}_0(\xi)$, $\hat{K}^{21}(\xi, t) \hat{u}_0(\xi) $ and $\hat{K}^{22}(\xi, t) \hat{\mathbb{E}}_0(\xi)  $ are defined by
\begin{align*}
\hat{K}^{12}(\xi, t) \hat{\mathbb{E}}_0(\xi)
=& \frac{  e^{\lambda_1 t}-  e^{\lambda_2 t}}{\lambda_1-\lambda_2}
\left[\left(\mathbb{I}-\frac{\xi \xi^{\top}}{|\xi|^2}\right) i \hat{\mathbb{E}}_0 \xi\right]
;
\end{align*}
\begin{align*}
\hat{K}^{21}(\xi, t) \hat{u}_0(\xi)
=& i \frac{  e^{\lambda_1 t}-  e^{\lambda_2 t}}{\lambda_1-\lambda_2}
\hat{u}_0 \xi^{\top} ;
\end{align*}
\begin{align*}
\hat{K}^{22}(\xi, t) \hat{\mathbb{E}}_0(\xi)
=& \frac{ \lambda_1 e^{\lambda_2 t}- \lambda_2 e^{\lambda_1 t}}{\lambda_1-\lambda_2}
\left[\left(\mathbb{I}-\frac{\xi \xi^{\top}}{|\xi|^2}\right)  \hat{\mathbb{E}}_0 \frac{\xi \xi^{\top}}{|\xi|^2} \right]
+\hat{\mathbb{E}}_0
-\left[\left(\mathbb{I}-\frac{\xi \xi^{\top}}{|\xi|^2}\right)  \hat{\mathbb{E}}_0 \frac{\xi \xi^{\top}}{|\xi|^2} \right]
;
\end{align*}
where $\lambda_j(\xi), j=1,2$, are given by
\begin{align*}
& \lambda_1(\xi)=\frac{- \mu |\xi|^2 + \sqrt{ \mu^2 |\xi|^4-4 |\xi|^2}}{2} ,\\
& \lambda_2(\xi)=\frac{- \mu |\xi|^2 - \sqrt{ \mu^2 |\xi|^4-4 |\xi|^2}}{2} .
\end{align*}

\end{Lemma}

\begin{proof}
Considering the linearized system \eqref{incompressible viscoelastic flow-E-202512201435}, a series of calculations yields:
\begin{align}\label{divf-202512261701}
\partial_{t t} u-\mu \Delta \partial_t u- \Delta u=0,
\end{align}
where $u(x, t)=\Delta^{-1} \operatorname{div}\operatorname{curl} u(x, t)$.
The Fourier transform of \eqref{divf-202512261701} is
\begin{align}\label{divfF-202512261702}
& \partial_{t t} \hat{d}+\mu|\xi|^2 \partial_t \hat{d}+|\xi|^2 \hat{d}=0.
\end{align}
By direct calculation, we obtain the eigenvalues:
\begin{align}\label{eigenvalues34-202512261703}
\lambda_{1,2}=-\frac{\mu}{2}|\xi|^2 \pm i |\xi| \frac{\sqrt{4-\mu^2|\xi|^2}}{2}.
\end{align}
Here $\lambda_{1,2}$ are the roots of the equation:
\begin{align}\label{roots34--202512261704}
\lambda^2+\mu|\xi|^2 \lambda+\beta^2|\xi|^2=0.
\end{align}
Let us consider the initial value problem for the equation \eqref{divf-202512261701}:
\begin{align}\label{IVPdiv-202512261705}
\left\{
\begin{array}{lll}
\partial_{t t} u-\mu \Delta \partial_t u- \Delta u=0,\\
u_0=u(x, 0), \quad u_1=\left.\partial_t u(x, t)\right|_{t=0},
\end{array}
\right.
\end{align}
with initial data of the form:
\begin{align}\label{initialdiv--202512261706}
\left\{
\begin{array}{lll}
u_0=u(x, 0), \\
u_1=\left.\partial_t u(x, t)\right|_{t=0}=\mu \Delta u_0
+\left(\Delta^{-1} \operatorname{divcurl}\right) \operatorname{div} \mathbb{E}_0.
\end{array}
\right.
\end{align}
Consequently, we obtain
\begin{align*}
\hat{u}(\xi, t)
&=
\frac{ \lambda_1 e^{\lambda_1 t}- \lambda_2 e^{\lambda_2 t}}{\lambda_1-\lambda_2}\hat{u}_0(\xi)
+
\frac{  e^{\lambda_1 t}-  e^{\lambda_2 t}}{\lambda_1-\lambda_2}
\left[\left(\mathbb{I}-\frac{\xi \xi^{\top}}{|\xi|^2}\right) i \hat{\mathbb{E}}_0 \xi\right].
\end{align*}
\begin{align*}
\hat{\mathbb{E}}(\xi, t)
&=\hat{\mathbb{E}}_0(\xi)+i(\int_0^t \hat{w}(\xi, \tau) \mathrm{d} \tau) \xi^{\top}
\\
&=
i \frac{  e^{\lambda_1 t}-  e^{\lambda_2 t}}{\lambda_1-\lambda_2}
\hat{u}_0 \xi^{\top}
+\frac{ \lambda_1 e^{\lambda_2 t}- \lambda_2 e^{\lambda_1 t}}{\lambda_1-\lambda_2}
\left[\left(\mathbb{I}-\frac{\xi \xi^{\top}}{|\xi|^2}\right)  \hat{\mathbb{E}}_0 \frac{\xi \xi^{\top}}{|\xi|^2} \right]
+\hat{\mathbb{E}}_0
-\left[\left(\mathbb{I}-\frac{\xi \xi^{\top}}{|\xi|^2}\right)  \hat{\mathbb{E}}_0 \frac{\xi \xi^{\top}}{|\xi|^2} \right].
\end{align*}
\end{proof}
The solution $S(t)=e^{-t L} S_0$ is thus given by
\begin{align*}
S(t)=e^{-t L} S_0=\mathcal{F}^{-1} e^{-t \hat{L}_{\xi}} \hat{S}_0.
\end{align*}
We will employ the following properties of $\lambda_j(j=1,2)$ to characterize the asymptotic behavior of $S(t)$:
\begin{align*}
\lambda_j(\xi)^2+ \mu |\xi|^2 \lambda_j(\xi)+ |\xi|^2=0,\quad  j=1,2, \\
\lambda_j(\xi) \sim-\frac{\mu}{2}|\xi|^2+i(-1)^{j+1} |\xi|, \quad \text { for }|\xi| \ll 1, \quad j=1,2, \\
\lambda_1(\xi) \sim-\frac{1}{\mu}, \mu_2(\xi) \sim-\mu|\xi|^2, \quad  \text { for }|\xi| \gg 1. \\
\end{align*}

We carry out a low-high frequency decomposition of the solution $S(t)$ for system \eqref{linearized-system-around-non-202512201510}.
Let $\hat{\varphi}_1$ and $\hat{\varphi}_{\infty}$ be $C^{\infty}(\mathbb{R}^3)$ cut-off functions satisfying
\begin{align*}
\hat{\varphi}_1(\xi)=\left\{\begin{array}{ll}
1 & |\xi| \leq \frac{M_1}{2}, \\
0 & |\xi| \geq \frac{M_1}{\sqrt{2}},\\
\end{array}\quad
\hat{\varphi}_1(-\xi)=\hat{\varphi}_1(\xi),\right.\quad
\hat{\varphi}_{\infty}(\xi)=1-\hat{\varphi}_1(\xi),
\quad
where M_1=\frac{1}{\mu}.
\end{align*}
Define $L^2$ operators $P_1$ and $P_{\infty}$ by
\begin{align*}
P_1 u=\mathcal{F}^{-1}\left(\hat{\varphi}_1 \hat{u}\right),
\quad
 P_{\infty} u=\mathcal{F}^{-1}\left(\hat{\varphi}_{\infty} \hat{u}\right) \text { for } u \in L^2 .
\end{align*}

We summarize the key properties of $P_j(j=1, \infty)$.
\begin{Lemma}\label{properties-P-202512201632}
It has the following properties:\\
(i) $P_1+P_{\infty}=I$.\\
(ii) $\partial_x^\alpha P_1=P_1 \partial_x^\alpha$, $\left\|\partial_x^\alpha P_1 f\right\|_{L^2} \leq C_\alpha\|f\|_{L^2}$ for $\alpha \in(\{0\} \cup \mathbb{N})^3$ and $f \in L^2$.\\
(iii) $\partial_x^\alpha P_{\infty}=P_{\infty} \partial_x^\alpha,\left\|\partial_x^\alpha P_{\infty} f\right\|_{L^2} \leq C\left\|\nabla \partial_x^\alpha P_{\infty} f\right\|_{L^2}$ for $\alpha \in(\{0\} \cup \mathbb{N})^3$ with $|\alpha|=k \geq 0$ and $f \in H^{k+1}$.
\end{Lemma}
\begin{proof}
Assertions (i), (ii), and (iii) follow from standard applications of the Plancherel theorem; the proofs are straightforward, and we omit the details.
\end{proof}

The solution $S(t)$ of \eqref{linearized-system-around-non-202512201510} admits the decomposition
\begin{align*}
S(t)=S_1(t)+S_{\infty}(t), S_1(t)=P_1 S(t), S_{\infty}(t)=P_{\infty} S(t).
\end{align*}

It follows that $S_1(t) := (u_1(t),\mathbb{E}1(t))$ and $S\infty(t) := (u_\infty(t),\mathbb{E}_\infty(t))$ are governed by the equations
\begin{align}\label{low-equation-202512201638}
\left\{\begin{array}{l}
S_1(t)=e^{-t L} S_1(0)+\int_0^t e^{-(t-\tau) L} P_1 N(\tau) \mathrm{d} \tau ,\\
\operatorname{div} u_1=0,
\end{array}\right.
\end{align}
and
\begin{align}\label{high-equation-202512201639}
\left\{\begin{array}{l}
\partial_t S_{\infty}+L S_{\infty}=P_{\infty} N ,\\
\operatorname{div} u_{\infty}=0.
\end{array}\right.
\end{align}

Lemma \ref{expression-202512201545} yields explicit expressions for $u$ and $\tilde{  \mathbb{E} }_c$, corresponding to the solutions of \eqref{linearized-system-around-non-202512201510}.
\begin{align}\label{u-linear-nolinear-202512221604}
u(x, t)=u_g(x, t)+u_n(x, t),
\end{align}
\begin{align}\label{u-g-linear-202512201817}
&( \mathcal{F} u_g(x, t) )(\xi ,t)\nonumber
=
\hat{u}_g(\xi, t)\\ \nonumber
&\triangleq
\hat{K}^{11}(\xi, t) \hat{u}_0(\xi)
+ \hat{K}^{12}(\xi, t) \hat{\mathbb{E}}_0(\xi)\\
&=
\frac{ \lambda_1 e^{\lambda_1 t}- \lambda_2 e^{\lambda_2 t}}{\lambda_1-\lambda_2}\hat{u}_0(\xi)
+ \frac{  e^{\lambda_1 t}-  e^{\lambda_2 t}}{\lambda_1-\lambda_2}
\left[\left(\mathbb{I}-\frac{\xi \xi^{\top}}{|\xi|^2}\right) i \hat{\mathbb{E}}_0 \xi\right],
\end{align}
\begin{align}\label{u-non-202512201816}
&( \mathcal{F} u_n(x, t) )(\xi ,t)\nonumber
=
\hat{u}_n(\xi, t)\\ \nonumber
& \triangleq \int_0^t
\hat{K}^{11}(\xi, t-\tau) \hat{N}_1(\xi, \tau)
+ \hat{K}^{12}(\xi, t-\tau) \hat{N}_2(\xi, \tau)
d\tau\\
&= \int_0^t
\frac{ \lambda_1 e^{\lambda_1 (t-\tau) }- \lambda_2 e^{\lambda_2 (t-\tau)}}{\lambda_1-\lambda_2} \hat{N}_1(\xi, \tau)
+
\frac{  e^{\lambda_1 (t-\tau)}-  e^{\lambda_2 (t-\tau)}}{\lambda_1-\lambda_2}
\left[\left(\mathbb{I}-\frac{\xi \xi^{\top}}{|\xi|^2}\right) i \hat{N}_2(\xi, \tau) \xi\right]
d\tau.
\end{align}
\begin{align}\label{EC-linear-nolinear-202512251804}
\tilde{  \mathbb{E} }_c(x, t)=\tilde{  \mathbb{E} }_{cg}(x, t)+\tilde{  \mathbb{E} }_{cn}(x, t),
\end{align}
\begin{align}\label{EC-g-linear-202512251806}
&( \mathcal{F} \tilde{  \mathbb{E} }_{cg}(x, t) )(\xi ,t)\nonumber
=
\hat{ \tilde{  \mathbb{E} } }_{cg}(\xi, t)\\
&\triangleq
i \frac{  e^{\lambda_1 t}-  e^{\lambda_2 t}}{\lambda_1-\lambda_2}
\hat{u}_0 \xi^{\top}
+ \frac{ \lambda_1 e^{\lambda_2 t}- \lambda_2 e^{\lambda_1 t}}{\lambda_1-\lambda_2}
\left[  \hat{\mathbb{E}}_0 \frac{\xi \xi^{\top}}{|\xi|^2} \right],
\end{align}
\begin{align}\label{EC-non-202512251806}
&( \mathcal{F} \tilde{  \mathbb{E} }_{cn}(x, t) )(\xi ,t)\nonumber
=
\hat{ \tilde{  \mathbb{E} } }_{cn}(\xi, t)\\
&= \int_0^t
i \frac{  e^{\lambda_1 (t-\tau)}-  e^{\lambda_2 (t-\tau)}}{\lambda_1-\lambda_2}
\hat{N}_1(\xi, \tau) \xi^{\top}
+
\frac{ \lambda_1 e^{\lambda_2 (t-\tau)}- \lambda_2 e^{\lambda_1 (t-\tau)}}{\lambda_1-\lambda_2}
\left[  \hat{N}_2(\xi, \tau) \frac{\xi \xi^{\top}}{|\xi|^2} \right]
d\tau.
\end{align}

We begin by analyzing the $L^\infty$ decay of the low-frequency parts $u_1(t)$ and $\tilde{\mathbb{E}}_{c1}(t)$.

\begin{Lemma}\label{Linfty-estimate-lowfrequency-S1-202512201644}
The following inequality holds for all $t \geq 0$, provided that there exists a positive number $\delta_0$ such that $\left\|S_0\right\|_{L^1}+\left\|S_0\right\|_{H^3} \leq \delta_0$:
\begin{align}\label{Linfty-estimate-lowfrequency-S1-decay-202512201645}
\left\|u_1(t)\right\|_{L^{\infty}} \leq C(1+t)^{-2}\left(\left\|S_0\right\|_{L^1}+\left\|S_0\right\|_{H^3}\right),
\end{align}
\begin{align}\label{EC-Linfty-estimate-lowfrequency-S1-decay-202512251835}
\left\|\tilde{  \mathbb{E} }_{c1}(t)\right\|_{L^{\infty}} \leq C(1+t)^{-2}\left(\left\|S_0\right\|_{L^1}+\left\|S_0\right\|_{H^3}\right).
\end{align}
\end{Lemma}

A central tool for proving Lemma \ref{Linfty-estimate-lowfrequency-S1-202512201644} is the following estimate.
\begin{Lemma}\label{low-Linfty-202512201649}
Under the conditions that $f \in L^1$, $\alpha \in(\mathbb{N} \cup\{0\})^3, j \in\{0\} \cup \mathbb{N}$ and $t>0$, the following estimates hold:
\begin{align*}
& \left\|\partial_t^j \partial_x^\alpha \mathcal{F}^{-1}\left[\frac{e^{\lambda_1(\xi) t}-e^{\lambda_2(\xi) t}}{\lambda_1(\xi)-\lambda_2(\xi)} \hat{\psi}(\xi) \hat{\varphi}_1(\xi)\right]\right\|_{L^{\infty}} \leq C(1+t)^{-\frac{3}{2}-\frac{j+|\alpha|}{2}},\\
& \left\|\partial_t^j \partial_x^\alpha \mathcal{F}^{-1}\left[\frac{\lambda_1(\xi) e^{\lambda_2(\xi) t}-\lambda_2(\xi) e^{\lambda_1(\xi) t}}{\lambda_1(\xi)-\lambda_2(\xi)} \hat{\psi}(\xi) \hat{\varphi}_1(\xi)\right]\right\|_{L^{\infty}} \leq C(1+t)^{-2-\frac{j+|\alpha|}{2}},
\end{align*}
where $\hat{\psi}(\xi)=\tilde{\psi}\left(\frac{\xi}{|\xi|}\right)$ with $\tilde{\psi} \in C^{\infty}(S^2)$ and $S^2=\{\omega \in \mathbb{R}^3|| \omega \mid=1\}$.
\end{Lemma}

\begin{proof}
We refer the reader to \cite[Theorem 3.1]{Kobayashi-Shibata2002} for the proof.
\end{proof}

We state the $L^\infty$ estimates for $u_{g1}(t)$ and $\tilde{\mathbb{E}}_{cg1}(t)$ as follows:
\begin{Lemma}\label{etLU1-low-Linfty-202512201654}
The following inequality holds for all $t \geq 0$, provided that there exists a positive number $\delta_0$ such that $\left\|S_0\right\|_{L^1}+\left\|S_0\right\|_{H^3} \leq \delta_0$:
\begin{align*}
\left\| u_{g1}(t) \right\|_{L^{\infty}} \leq C(1+t)^{-2}\left\|S_0\right\|_{L^1} .
\end{align*}
\begin{align*}
\left\| \tilde{  \mathbb{E} }_{cg1}(t) \right\|_{L^{\infty}} \leq C(1+t)^{-2}\left\|S_0\right\|_{L^1} .
\end{align*}
\end{Lemma}

Lemma \ref{etLU1-low-Linfty-202512201654} follows directly from Lemma \ref{low-Linfty-202512201649}.
We estimate $\left\| u_{n1}(t) \right\|_{L^{\infty}}$ as follows:

\begin{Lemma}\label{ettauLP1N-202512201700}
The following inequality holds for all $t \geq 0$, provided that there exists a positive number $\delta_0$ such that $\left\|S_0\right\|_{L^1}+\left\|S_0\right\|_{H^3} \leq \delta_0$:
\begin{align*}
\left\| u_{n1}(t) \right\|_{L^{\infty}} \leq C(1+t)^{-2}\left(\left\|S_0\right\|_{L^1}+\left\|S_0\right\|_{H^3}\right),
\end{align*}
\begin{align*}
\left\| \tilde{  \mathbb{E} }_{cn1}(t) \right\|_{L^{\infty}} \leq C(1+t)^{-2}\left(\left\|S_0\right\|_{L^1}+\left\|S_0\right\|_{H^3}\right).
\end{align*}
\end{Lemma}

\begin{proof}
Since
\begin{align}\label{202512261618-1}
\left\|N_1(\tau)\right\|_{L^1}
\leq C\|S(\tau)\|_{L^2}\|\nabla S(\tau)\|_{L^2}
\leq C(1+\tau)^{-2},
\end{align}
\begin{align}\label{202512261618-2}
\left\|N_2(\tau)\right\|_{L^1}
\leq C\|S(\tau)\|_{L^2}\|\nabla S(\tau)\|_{L^2}
\leq C(1+\tau)^{-2},
\end{align}
combining Lemma \ref{low-Linfty-202512201649} with \eqref{202512261618-1} and \eqref{202512261618-2} yields the estimates \eqref{202512261621-3}-\eqref{202512261622-4}:
\begin{align}\label{202512261621-3}
\left\|\mathcal{F}^{-1}\left[\hat{\varphi}_1(\xi) \hat{K}^{11}(\xi, t-\tau) \hat{N}_1(\xi, \tau)\right]\right\|_{L^{\infty}}
\leq C(1+t-\tau)^{-2}(1+\tau)^{-2},
\end{align}
\begin{align}\label{202512261622-4}
\left\|\mathcal{F}^{-1}\left[\hat{\varphi}_1(\xi) \hat{K}^{12}(\xi, t-\tau) \hat{N}_2(\xi, \tau)\right]\right\|_{L^{\infty}}
\leq C(1+t-\tau)^{-2}(1+\tau)^{-2}.
\end{align}
In summary, together with Lemma \ref{ele}, estimates \eqref{202512261621-3}-\eqref{202512261622-4} lead to
\begin{align*}
&\quad \|u_{n1}(t)\|_{L^{\infty}}
=
\|
\mathcal{F}^{-1}\left[
u_n(\xi, t)
\hat{\varphi}_1(\xi)
\right]
\|_{L^{\infty}}\\
& \leq
\int_0^t
\left\|\mathcal{F}^{-1}\left[\hat{\varphi}_1(\xi) \hat{K}^{11}(\xi, t-\tau) \hat{N}_1(\xi, \tau)\right]\right\|_{L^{\infty}}
d\tau+
\int_0^t
\left\|\mathcal{F}^{-1}\left[\hat{\varphi}_1(\xi) \hat{K}^{12}(\xi, t-\tau) \hat{N}_2(\xi, \tau)\right]\right\|_{L^{\infty}}
d\tau
\\
& \leq C \int_0^t(1+t-\tau)^{-2}(1+\tau)^{-2} d\tau
+C\int_0^t  (1+t-\tau)^{-2}(1+\tau)^{-2}  d\tau \\
& \leq C(1+t)^{-2}.
\end{align*}
The remaining cases can be proved in a similar manner.
\end{proof}

\noindent{\bf Proof of Lemma \ref{Linfty-estimate-lowfrequency-S1-202512201644}  }
The $L^{\infty}$ norm applied to \eqref{u-linear-nolinear-202512221604} gives
\begin{align}\label{Linfty-norm-firstequation-202512201707}
\left\|u_1(t)\right\|_{L^{\infty}} \leq
\left\| u_{g1}(t) \right\|_{L^{\infty}}
+\left\| u_{n1}(t) \right\|_{L^{\infty}}.
\end{align}
Together with Lemmas \ref{etLU1-low-Linfty-202512201654} and \ref{ettauLP1N-202512201700}, equation \eqref{Linfty-norm-firstequation-202512201707} leads to
\begin{align}\label{S1-202512201709}
\left\|u_1(t)\right\|_{L^{\infty}} \leq C(1+t)^{-2}\left(\left\|S_0\right\|_{L^1}+\left\|S_0\right\|_{H^3}\right) .
\end{align}
The remaining cases can be proved analogously. This completes the proof of Lemma \ref{Linfty-estimate-lowfrequency-S1-202512201644}.

We proceed to the analysis of the high-frequency components $u_\infty(t)$ and $\tilde{\mathbb{E}}_{c\infty}$.

\begin{Lemma}\label{high frequency part-S-infty-202512201712}
The following inequality holds for all $t \geq 0$, provided that there exists a positive number $\delta_0$ such that $\left\|S_0\right\|_{L^1}+\left\|S_0\right\|_{H^3} \leq \delta_0$:
\begin{align*}
\left\|u_{\infty}(t)\right\|_{L^{\infty}} \leq C(1+t)^{-2}\left(\left\|S_0\right\|_{L^1}+\left\|S_0\right\|_{H^3}\right),
\end{align*}
\begin{align*}
\left\|\tilde{  \mathbb{E} }_{c\infty}(t)\right\|_{L^{\infty}} \leq C(1+t)^{-2}\left(\left\|S_0\right\|_{L^1}+\left\|S_0\right\|_{H^3}\right).
\end{align*}
\end{Lemma}

\begin{proof}
Taken together, Lemma \ref{properties-P-202512201632} and Theorem \ref{pure-energy-202512221702} can be combined to complete the proof of the lemma:
\begin{align*}
\|u_{\infty}(t)\|_{L^{\infty}}
\leq
\|u_{\infty}(t)\|_{H^{2}}
\leq
\| \nabla^{3} u_{\infty}(t)\|_{L^{2}}
\leq
\| \nabla^{3} u(t)\|_{L^{2}}
\leq
C(1+t)^{-\frac{3}{4}-\frac{3}{2}},
\end{align*}
\begin{align*}
\|\tilde{  \mathbb{E} }_{c\infty}(t)\|_{L^{\infty}}
\leq
\|\tilde{  \mathbb{E} }_{c\infty}(t)\|_{H^{2}}
\leq
\| \nabla^{3} \tilde{  \mathbb{E} }_{c\infty}(t)\|_{L^{2}}
\leq
\| \nabla^{3} \tilde{  \mathbb{E} }_{c}(t)\|_{L^{2}}
\leq
\| \nabla^{3} \mathbb{E}(t)\|_{L^{2}}
\leq
C(1+t)^{-\frac{3}{4}-\frac{3}{2}}.
\end{align*}
\end{proof}

\noindent{\bf Proof of Lemma \ref{Linfty-pure-energy-202512201519}  }
Lemma \ref{Linfty-pure-energy-202512201519} follows directly from Lemmas \ref{low-Linfty-202512201649} and \ref{high frequency part-S-infty-202512201712}, thereby completing its proof.

\section{Proof of Theorem \ref{diffusion waves-1-202512201453} }\label{202512201720-Theorem-proof}

This section is devoted to the proof of Theorem \ref{diffusion waves-1-202512201453}. The global existence and higher-order $L^2$ decay are already established by Proposition \ref{2015Hu-Wu-202512201444}. We therefore turn to the $L^p$ decay estimates for $p \neq 2$.
By the interpolation inequality: $\left\|u(t)\right\|_{L^p} \leq\left\|u(t)\right\|_{L^1}^{\frac{2}{p}-1}\left\|u(t)\right\|_{L^2}^{2-\frac{2}{p}}, 1\leq p<2$, the $L^p$ decay problem reduces to establishing the $L^1$ decay of $u$ and $\tilde{\mathbb{E}}_c$.

We prove Theorem \ref{diffusion waves-1-202512201453} via a combination of Lemma \ref{L2-decay-estimates-202512251734} and the $L^1$ decay estimates below for $u(t)$ and $\tilde{\mathbb{E}}_c$.
\begin{Lemma}\label{SL1-202512201804}
The following inequality holds for all $t \geq 0$, provided that there exists a positive number $\delta_0$ such that $\left\|S_0\right\|_{L^1}+\left\|S_0\right\|_{H^3} \leq \delta_0$:
\begin{align*}
\left\|u(t)\right\|_{L^1} \leq C(1+t)^{\frac{1}{2}}
\left(\left\|S_0\right\|_{L^1}+\left\|S_0\right\|_{H^3}\right),
\end{align*}
\begin{align*}
\left\|\tilde{  \mathbb{E} }_c(t)\right\|_{L^1} \leq C(1+t)^{\frac{1}{2}}
\left(\left\|S_0\right\|_{L^1}+\left\|S_0\right\|_{H^3}\right).
\end{align*}
\end{Lemma}

We begin by analyzing the $L^1$ decay of the low-frequency parts $u_1(t)$ and $\tilde{\mathbb{E}}_{c1}$.
\begin{Lemma}\label{SL1-low-202512201805}
The following inequality holds for all $t \geq 0$, provided that there exists a positive number $\delta_0$ such that $\left\|S_0\right\|_{L^1}+\left\|S_0\right\|_{H^3} \leq \delta_0$:
\begin{align*}
\left\|u_1(t)\right\|_{L^1} \leq C(1+t)^{\frac{1}{2}}
\left(\left\|S_0\right\|_{L^1}+\left\|S_0\right\|_{H^3}\right),
\end{align*}
\begin{align*}
\left\|\tilde{  \mathbb{E} }_{c1}(t)\right\|_{L^1} \leq C(1+t)^{\frac{1}{2}}
\left(\left\|S_0\right\|_{L^1}+\left\|S_0\right\|_{H^3}\right).
\end{align*}
\end{Lemma}
The proof of Lemma \ref{SL1-low-202512201805} relies on the following auxiliary result.
\begin{Lemma}\label{lowlemmaL1-202512221537}
Under the conditions that $f \in L^1$, $\alpha \in(\mathbb{N} \cup\{0\})^3, j \geq 0$ and $t>0$, the following estimates hold:
\begin{align*}
& \left\|\partial_t^j \partial_x^\alpha \mathcal{F}^{-1}\left[\frac{e^{\lambda_1(\xi) t}-e^{\lambda_2(\xi) t}}{\lambda_1(\xi)-\lambda_2(\xi)} \hat{\varphi}_1(\xi)\right]\right\|_{L^1} \leq C(1+t)^{1-\frac{j+|\alpha|}{2}}, \\
& \left\|\partial_t^j \partial_x^\alpha \mathcal{F}^{-1}\left[\frac{\lambda_1(\xi) e^{\lambda_2(\xi) t}-\lambda_2(\xi) e^{\lambda_1(\xi) t}}{\lambda_1(\xi)-\lambda_2(\xi)} \hat{\varphi}_1(\xi)\right]\right\|_{L^1} \leq C(1+t)^{\frac{1}{2}-\frac{j+|\alpha|}{2}}.
\end{align*}
\end{Lemma}
\begin{proof}
We refer the reader to \cite[p.216]{Shibata2000} for the proof.
\end{proof}

We state the $L^1$ estimates for $u_{g1}(t)$ and $\tilde{\mathbb{E}}_{cg1}(t)$ as follows:
\begin{Lemma}\label{etLU1-low-1-202512221541}
The following inequality holds for all $t \geq 0$, provided that there exists a positive number $\delta_0$ such that $\left\|S_0\right\|_{L^1}+\left\|S_0\right\|_{H^3} \leq \delta_0$:
\begin{align*}
\left\|  u_{g1}(t) \right\|_{L^{1}} \leq C(1+t)^{ \frac{1}{2} }\left\|S_0\right\|_{L^1},
\end{align*}
\begin{align*}
\left\|  \tilde{  \mathbb{E} }_{cg1}(t) \right\|_{L^{1}} \leq C(1+t)^{ \frac{1}{2} }\left\|S_0\right\|_{L^1}.
\end{align*}
\end{Lemma}
The proof of Lemma \ref{etLU1-low-1-202512221541} reduces to an application of Lemma \ref{lowlemmaL1-202512221537}.

We obtain the following $L^1$ estimates for $u_{n1}(t)$ and $\tilde{\mathbb{E}}_{cn1}(t)$:

\begin{Lemma}\label{ettauLP1N-L1-202512221602}
The following inequality holds for all $t \geq 0$, provided that there exists a positive number $\delta_0$ such that $\left\|S_0\right\|_{L^1}+\left\|S_0\right\|_{H^3} \leq \delta_0$:
\begin{align*}
\left\| u_{n1}(t) \right\|_{L^{1}} \leq C(1+t)^{ \frac{1}{2} } \left(\left\|S_0\right\|_{L^1}+\left\|S_0\right\|_{H^3}\right),
\end{align*}
\begin{align*}
\left\| \tilde{  \mathbb{E} }_{cn1}(t) \right\|_{L^{1}} \leq C(1+t)^{ \frac{1}{2} } \left(\left\|S_0\right\|_{L^1}+\left\|S_0\right\|_{H^3}\right).
\end{align*}
\end{Lemma}

\begin{proof}
Since
\begin{align}\label{202512261634-1}
\left\|N_1(\tau)\right\|_{L^1}
\leq C\|S(\tau)\|_{L^2}\|\nabla S(\tau)\|_{L^2}
\leq C(1+\tau)^{-2},
\end{align}
\begin{align}\label{202512261634-2}
\left\|N_2(\tau)\right\|_{L^1}
\leq C\|S(\tau)\|_{L^2}\|\nabla S(\tau)\|_{L^2}
\leq C(1+\tau)^{-2},
\end{align}
from Lemma \ref{lowlemmaL1-202512221537} together with \eqref{202512261634-1} and \eqref{202512261634-2}, we obtain estimates \eqref{202512261634-3}-\eqref{202512261634-4}:
\begin{align}\label{202512261634-3}
\left\|\mathcal{F}^{-1}\left[\hat{\varphi}_1(\xi) \hat{K}^{11}(\xi, t-\tau) \hat{N}_1(\xi, \tau)\right]\right\|_{L^{1}}
\leq C(1+t-\tau)^{\frac{1}{2}}(1+\tau)^{-2},
\end{align}
\begin{align}\label{202512261634-4}
\left\|\mathcal{F}^{-1}\left[\hat{\varphi}_1(\xi) \hat{K}^{12}(\xi, t-\tau) \hat{N}_2(\xi, \tau)\right]\right\|_{L^{1}}
\leq C(1+t-\tau)^{\frac{1}{2}}(1+\tau)^{-2}.
\end{align}
To summarize, \eqref{202512261634-3}-\eqref{202512261634-4} together with Lemma \ref{ele} yield
\begin{align*}
&\quad \|u_{n1}(t)\|_{L^{\infty}}
=
\|
\mathcal{F}^{-1}\left[
u_n(\xi, t)
\hat{\varphi}_1(\xi)
\right]
\|_{L^{\infty}}\\
& \leq
\int_0^t
\left\|\mathcal{F}^{-1}\left[\hat{\varphi}_1(\xi) \hat{K}^{11}(\xi, t-\tau) \hat{N}_1(\xi, \tau)\right]\right\|_{L^{\infty}}
d\tau+
\int_0^t
\left\|\mathcal{F}^{-1}\left[\hat{\varphi}_1(\xi) \hat{K}^{12}(\xi, t-\tau) \hat{N}_2(\xi, \tau)\right]\right\|_{L^{\infty}}
d\tau
\\
& \leq C \int_0^t(1+t-\tau)^{\frac{1}{2}}(1+\tau)^{-2} d\tau
+C\int_0^t  (1+t-\tau)^{\frac{1}{2}}(1+\tau)^{-2}  d\tau \\
& \leq C(1+t)^{\frac{1}{2}}.
\end{align*}
The remaining cases follow straightforwardly from the arguments above.
\end{proof}

\noindent{\bf Proof of Lemma \ref{SL1-low-202512201805}  }
An application of the $L^1$ norm to \eqref{u-linear-nolinear-202512221604} gives
\begin{align}\label{L1-norm-firstequation-202512221609}
\left\|u_1(t)\right\|_{L^{1}} \leq
\left\| u_{g1}(t) \right\|_{L^{1}}
+\left\| u_{n1}(t) \right\|_{L^{1}}.
\end{align}
From \eqref{L1-norm-firstequation-202512221609}, together with Lemmas \ref{etLU1-low-1-202512221541} and \ref{ettauLP1N-L1-202512221602}, we obtain
\begin{align}\label{S1-L1-202512221611}
\left\|u_1(t)\right\|_{L^{1}} \leq C(1+t)^{ \frac{1}{2} }\left(\left\|S_0\right\|_{L^1}+\left\|S_0\right\|_{H^3}\right) .
\end{align}
The remaining cases follow similarly, which completes the proof of Lemma \ref{SL1-low-202512201805}.

We proceed to analyze the high-frequency components $u_\infty(t)$ and $\tilde{\mathbb{E}}_{c\infty}(t)$.
\begin{Lemma}\label{SL1-high-202512201806}
The following inequality holds for all $t \geq 0$, provided that there exists a positive number $\delta_0$ such that $\left\|S_0\right\|_{L^1}+\left\|S_0\right\|_{H^3} \leq \delta_0$:
\begin{align*}
\left\|u_{\infty}(t)\right\|_{L^1} \leq C(1+t)^{ -2 }
\left(\left\|S_0\right\|_{L^1}+\left\|S_0\right\|_{H^3}\right),
\end{align*}
\begin{align*}
\left\| \tilde{  \mathbb{E} }_{c\infty} (t)\right\|_{L^1} \leq C(1+t)^{ -2 }
\left(\left\|S_0\right\|_{L^1}+\left\|S_0\right\|_{H^3}\right).
\end{align*}
\end{Lemma}

To prove Lemma \ref{SL1-high-202512201806}, we first state and prove a key estimate.
\begin{Lemma}\label{highlemmaL1-202512221615}
Under the conditions that $\alpha \in(\mathbb{N} \cup\{0\})^3, j \geq k \geq 0$ and $t>0$, the following estimates hold:
\begin{align*}
& \left\|\partial_t^j \partial_x^\alpha \mathcal{F}^{-1}\left[\frac{e^{\lambda_1(\xi) t}}{\lambda_1(\xi)-\lambda_2(\xi)} \hat{\varphi}_{\infty}(\xi) \hat{f}(\xi)\right]\right\|_{L^1}
\leq C e^{-c t}\|f\|_{W^{(|\alpha|-1)^{+}, 1}}
, \\
& \left\|\partial_t^j \partial_x^\alpha \mathcal{F}^{-1}\left[\frac{e^{\lambda_2(\xi) t}}{\lambda_1(\xi)-\lambda_2(\xi)} \hat{\varphi}_{\infty}(\xi) \hat{f}(\xi)\right]\right\|_{L^1} \leq C e^{-c t} t^{-(j-k)}\|f\|_{W^{2 k+(|\alpha|-1)^{+}, 1}}
,\\
& \left\|\partial_t^j \partial_x^\alpha \mathcal{F}^{-1}\left[\frac{\lambda_2(\xi) e^{\lambda_1(\xi) t}}{\lambda_1(\xi)-\lambda_2(\xi)} \hat{\varphi}_{\infty}(\xi) \hat{f}(\xi)\right]\right\|_{L^1} \leq C e^{-c t}\|f\|_{W^{|\alpha|, 1}},\\
& \left\|\partial_t^j \partial_x^\alpha \mathcal{F}^{-1}\left[\frac{\lambda_1(\xi) e^{\lambda_2(\xi) t}}{\lambda_1(\xi)-\lambda_2(\xi)} \hat{\varphi}_{\infty}(\xi) \hat{f}(\xi)\right]\right\|_{L^1} \leq C e^{-c t} t^{-(j-k)}\|f\|_{W^{2 k+(|\alpha|-1)^{+}, 1}},\\
& \left\|\partial_t^j \partial_x^\alpha \mathcal{F}^{-1}\left[\frac{e^{\lambda_1(\xi) t}}{\lambda_1(\xi)-\lambda_2(\xi)} \frac{\xi \xi^{\top}}{|\xi|^2} \hat{\varphi}_{\infty}(\xi) \hat{f}(\xi)\right]\right\|_{L^1} \leq C e^{-c t}\|f\|_{W^{|\alpha|, 1}}
,\\
& \left\|\partial_t^j \partial_x^\alpha \mathcal{F}^{-1}\left[\frac{e^{\lambda_2(\xi) t}}{\lambda_1(\xi)-\lambda_2(\xi)} \frac{\xi \xi^{\top}}{|\xi|^2} \hat{\varphi}_{\infty}(\xi) \hat{f}(\xi)\right]\right\|_{L^1} \leq C e^{-c t} t^{-(j-k)}\|f\|_{W^{2 k+|\alpha|, 1}}.
\end{align*}
Here $a^{+}$denotes $a^{+}=\max \{0, a\}$ for $a \in \mathbb{R}$.
\end{Lemma}
\begin{proof}
We refer the reader to \cite[Lemma 5.1]{Ishigaki2022} for the proof.
\end{proof}

We estimate $\left\| u_{g\infty}(t) \right\|_{L^{1}}$ and $\left\| \tilde{  \mathbb{E} }_{cg\infty}(t) \right\|_{L^{1}}$ as follows.
\begin{Lemma}\label{etLU1-high-L1-202512221620}
The following inequality holds for all $t \geq 0$, provided that there exists a positive number $\delta_0$ such that $\left\|S_0\right\|_{L^1}+\left\|S_0\right\|_{H^3} \leq \delta_0$:
\begin{align*}
\left\| u_{g\infty}(t) \right\|_{L^{1}} \leq Ce^{-c t}\left\|S_0\right\|_{L^1},
\end{align*}
\begin{align*}
\left\| \tilde{  \mathbb{E} }_{cg\infty} (t) \right\|_{L^{1}} \leq Ce^{-c t}\left\|S_0\right\|_{L^1}.
\end{align*}
\end{Lemma}

Lemma \ref{etLU1-high-L1-202512221620} follows directly from Lemma \ref{highlemmaL1-202512221615}. For simplicity, we set the initial norm $\left\|S_0\right\|_{\mathcal{X}}=\left\|S_0\right\|_{L^1}+\left\|S_0\right\|_{H^3}$. We now estimate $\left\| u_{n\infty}(t) \right\|_{L^{1}}$ and $\left\| \tilde{  \mathbb{E} }_{cn\infty} (t) \right\|_{L^{1}}$.
\begin{Lemma}\label{etLU1-high-L1-non-202512221636}
The following inequality holds for all $t \geq 0$, provided that there exists a positive number $\delta_0$ such that $\left\|S_0\right\|_{L^1}+\left\|S_0\right\|_{H^3} \leq \delta_0$:
\begin{align*}
\left\| u_{n\infty}(t) \right\|_{L^{1}} \leq C(1+t)^{-2}\left\|S_0\right\|_{L^1},
\end{align*}
\begin{align*}
\left\| \tilde{  \mathbb{E} }_{cn\infty}(t) \right\|_{L^{1}} \leq C(1+t)^{-2}\left\|S_0\right\|_{L^1}.
\end{align*}
\end{Lemma}

\begin{proof}
Since
\begin{align}\label{202512261639-1}
\left\|N_1(\tau)\right\|_{L^1}
\leq C\|S(\tau)\|_{L^2}\|\nabla S(\tau)\|_{L^2}
\leq C(1+\tau)^{-2},
\end{align}
\begin{align}\label{202512261639-2}
\left\|N_2(\tau)\right\|_{L^1}
\leq C\|S(\tau)\|_{L^2}\|\nabla S(\tau)\|_{L^2}
\leq C(1+\tau)^{-2},
\end{align}
from Lemma \ref{etLU1-high-L1-202512221620} together with \eqref{202512261639-1} and \eqref{202512261639-2}, we obtain estimates \eqref{202512261639-3}-\eqref{202512261639-4}:
\begin{align}\label{202512261639-3}
\left\|\mathcal{F}^{-1}\left[\hat{\varphi}_{\infty}(\xi) \hat{K}^{11}(\xi, t-\tau) \hat{N}_1(\xi, \tau)\right]\right\|_{L^{1}}
\leq Ce^{-c(t-s)}(1+\tau)^{-2},
\end{align}
\begin{align}\label{202512261639-4}
\left\|\mathcal{F}^{-1}\left[\hat{\varphi}_{\infty}(\xi) \hat{K}^{12}(\xi, t-\tau) \hat{N}_2(\xi, \tau)\right]\right\|_{L^{1}}
\leq Ce^{-c(t-s)}(1+\tau)^{-2}.
\end{align}
Putting together \eqref{202512261639-3}-\eqref{202512261639-4} and Lemma \ref{ele} leads to
\begin{align*}
&\quad \|u_{n\infty}(t)\|_{L^{\infty}}
=
\|
\mathcal{F}^{-1}\left[
u_n(\xi, t)
\hat{\varphi}_{\infty}(\xi)
\right]
\|_{L^{\infty}}\\
& \leq
\int_0^t
\left\|\mathcal{F}^{-1}\left[\hat{\varphi}_1(\xi) \hat{K}^{11}(\xi, t-\tau) \hat{N}_1(\xi, \tau)\right]\right\|_{L^{\infty}}
d\tau+
\int_0^t
\left\|\mathcal{F}^{-1}\left[\hat{\varphi}_1(\xi) \hat{K}^{12}(\xi, t-\tau) \hat{N}_2(\xi, \tau)\right]\right\|_{L^{\infty}}
d\tau
\\
& \leq C \int_0^te^{-c(t-s)}(1+\tau)^{-2} d\tau
+C\int_0^t  e^{-c(t-s)}(1+\tau)^{-2}  d\tau \\
& \leq C(1+t)^{-2}.
\end{align*}
The proofs of the remaining cases are analogous.
\end{proof}

\noindent{\bf Proof of Lemma \ref{SL1-high-202512201806}  }
An application of the $L^1$ norm to \eqref{u-linear-nolinear-202512221604} gives
\begin{align}\label{L1-norm-firstequation-202512221651}
\left\|u_{\infty}(t)\right\|_{L^{1}} \leq
\left\| u_{g\infty}(t) \right\|_{L^{1}}
+\left\| u_{n\infty}(t) \right\|_{L^{1}}.
\end{align}
Combining \eqref{L1-norm-firstequation-202512221651} with Lemmas \ref{etLU1-high-L1-202512221620} and \ref{etLU1-high-L1-non-202512221636} yields
\begin{align}\label{S1-L1-infty-202512221654}
\left\|u_\infty(t)\right\|_{L^{1}} \leq C(1+t)^{ -2 }\left(\left\|S_0\right\|_{L^1}+\left\|S_0\right\|_{H^3}\right) .
\end{align}
The remaining cases follow similarly, which completes the proof of Lemma \ref{SL1-high-202512201806}.

\noindent{\bf Proof of Lemma \ref{SL1-202512201804}  }
Proof of Lemma \ref{SL1-202512201804} follows directly from Lemma \ref{SL1-low-202512201805} and Lemma \ref{SL1-high-202512201806}, which completes the proof.

\section{Energy Estimates}\label{Energy Estimates-202512221729}
In this section, we aim to derive the necessary energy estimates.
Based on Theorem \ref{pure-energy-202512221702}, we proceed to derive the energy estimates for the solutions of \eqref{incompressible viscoelastic flow-E-202512201224}-\eqref{initial-condition-202512201225}. The first set of estimates concerns $u$ and $\mathbb{E}$.
\begin{Lemma}\label{one type of energy estimates including-202512221737}
Assume that
\begin{align*}
\|(u,\mathbb{E})\|_{H^3} \leq \delta \ll 1,
\end{align*}
and let the assumptions of Theorem \ref{pure-energy-202512221702} hold. Then, for any integer $0 \leq \ell \leq N-1$, we have
\begin{align}\label{one type of energy estimates including-202512221740}
\frac{d}{d t}\left\|\nabla^{\ell}(u,\mathbb{E})\right\|_{L^2}^2+C\left\|\nabla^{\ell+1} u\right\|_{L^2}^2 \lesssim \delta\left\|\nabla^{\ell+1}(u,\mathbb{E})\right\|_{L^2}^2.
\end{align}
\end{Lemma}
\begin{proof}
First, we consider the case $\ell = 0$. Multiplying equations $\eqref{incompressible viscoelastic flow-E-202512201224}_1$ and $\eqref{incompressible viscoelastic flow-E-202512201224}_2$ by $u$ and $\mathbb{E}$, respectively, summing the results, and integrating over $\mathbb{R}^3$, we apply integration by parts to obtain
\begin{align}\label{202512221744-1}
& \frac{1}{2} \frac{d}{d t} \int_{\mathbb{R}^3}|u|^2+|\mathbb{E}|^2 d x \nonumber
+\int_{\mathbb{R}^3} \mu|\nabla u|^2 d x\\
= & \int_{\mathbb{R}^3}
(\nabla u \mathbb{E} - u\cdot \nabla \mathbb{E} ): \mathbb{E} d x
+\int_{\mathbb{R}^3}
[\operatorname{div}(\mathbb{E} \mathbb{E}^{\top})
\left.-u \cdot \nabla u \right] u d x.
\end{align}
In the subsequent estimates, summation over repeated indices (Einstein convention) is implied.
By an application of Holder¡¯s inequality as well as the Sobolev inequality, we deduce
\begin{align}\label{202512221809-2}
& \int_{\mathbb{R}^3} \nonumber
(\nabla u \mathbb{E} - u\cdot \nabla \mathbb{E} ): \mathbb{E} d x
+\int_{\mathbb{R}^3}
[\operatorname{div}(\mathbb{E} \mathbb{E}^{\top})
\left.-u \cdot \nabla u \right] u d x \\
\lesssim & \|(u, \mathbb{E})\|_{L^3}\|(u, \mathbb{E})\|_{L^6}\|\nabla(u, \mathbb{E})\|_{L^2}
\lesssim  \delta\|\nabla(u, \mathbb{E})\|_{L^2}^2.
\end{align}
Consequently, \eqref{202512221809-2} immediately yields \eqref{one type of energy estimates including-202512221740} for the case $\ell=0$.
We proceed to examine the range $1 \leq \ell \leq N-1$. Acting with $\nabla^{\ell}$ on $\eqref{incompressible viscoelastic flow-E-202512201224}_1-\eqref{incompressible viscoelastic flow-E-202512201224}_2$, multiplying the resulting identities by
$\nabla^{\ell} u$ and $\nabla^{\ell} \mathbb{E}$, respectively, summing, and integrating over $\mathbb{R}^3$, we obtain upon integration by parts
\begin{align}\label{202512221825-3}
& \frac{1}{2} \frac{d}{d t} \int_{\mathbb{R}^3} \nonumber
\left|\nabla^{\ell} u\right|^2+\left|\nabla^{\ell} \mathbb{E}\right|^2 d x
+\int_{\mathbb{R}^3} \mu\left|\nabla^{\ell+1} u\right|^2 d x  \\ \nonumber
& =\int_{\mathbb{R}^3}
\nabla^{\ell}(\nabla u \mathbb{E}-u \cdot \nabla \mathbb{E}): \nabla^{\ell} \mathbb{E}
+\nabla^{\ell}
[\operatorname{div}(\mathbb{E} \mathbb{E}^{\top})
\left.-u \cdot \nabla u \right]
\cdot \nabla^{\ell} u d x \\
& =J_1+J_2.
\end{align}
Following the same procedure, we arrive at
\begin{align}\label{202512221833-4}
J_1=\int_{\mathbb{R}^3} \nabla^{\ell}(\nabla u \mathbb{E}-u \cdot \nabla \mathbb{E}): \nabla^{\ell} \mathbb{E} d x \lesssim \delta\left\|\nabla^{\ell+1}(u, \mathbb{E})\right\|_{L^2}^2,
\end{align}
\begin{align}\label{202512221834-5}
J_2 & =\int_{\mathbb{R}^3} \nabla^{\ell}
[\operatorname{div}(\mathbb{E} \mathbb{E}^{\top})
\left.-u \cdot \nabla u \right]
\cdot \nabla^{\ell} u d x
 \lesssim \delta\left\|\nabla^{\ell+1}(u, \mathbb{E})\right\|_{L^2}^2.
\end{align}
Having established \eqref{202512221833-4} and \eqref{202512221834-5}, the validity of \eqref{one type of energy estimates including-202512221740} for $1 \leq \ell \leq N-1$ follows.
\end{proof}
Next, we aim to establish a second set of energy estimates for $u$ and $\mathbb{E}$.
\begin{Lemma}\label{another type of energy estimates-202512221839-202512231850}
Assume that
\begin{align*}
\|(u,\mathbb{E})\|_{H^3} \leq \delta \ll 1,
\end{align*}
and let the assumptions of Theorem \ref{pure-energy-202512221702} hold. Then, for any integer $1 \leq \ell \leq N$, we have
\begin{align}\label{another type of energy estimates-202512221839}
\frac{d}{d t}\left\|\nabla^{\ell}(u,\mathbb{E})\right\|_{L^2}^2+C\left\|\nabla^{\ell+1} u\right\|_{L^2}^2 \lesssim \delta\left\|\nabla^{\ell}(u,\mathbb{E})\right\|_{L^2}^2.
\end{align}
\end{Lemma}

\begin{proof}
We now consider $1 \leq \ell \leq N$. Acting with $\nabla^{\ell}$ on $\eqref{incompressible viscoelastic flow-E-202512201224}_1-\eqref{incompressible viscoelastic flow-E-202512201224}_2$ and multiplying by $\nabla^{\ell} u$ and $\nabla^{\ell} \mathbb{E}$ respectively, summing, and integrating by parts over $\mathbb{R}^3$ gives
\begin{align}\label{202512221850-K1K2}
& \frac{1}{2} \frac{d}{d t} \int_{\mathbb{R}^3} \nonumber
\left|\nabla^{\ell} u\right|^2+\left|\nabla^{\ell} \mathbb{E}\right|^2 d x
+\int_{\mathbb{R}^3} \mu\left|\nabla^{\ell+1} u\right|^2 d x  \\ \nonumber
& =\int_{\mathbb{R}^3}
\nabla^{\ell}(\nabla u \mathbb{E}-u \cdot \nabla \mathbb{E}): \nabla^{\ell} \mathbb{E}
+\nabla^{\ell}
[\operatorname{div}(\mathbb{E} \mathbb{E}^{\top})
\left.-u \cdot \nabla u \right]
\cdot \nabla^{\ell} u d x \\
& =K_1+K_2+K_3+K_4.
\end{align}
To begin, we estimate the term $K_1$. Applying the Leibniz formula and Holder's inequality yields
\begin{align}\label{202512221850-K1K2-K1}
K_1 & =\int_{\mathbb{R}^3} \nabla^{\ell}
\nonumber
\left(\operatorname{div}(\mathbb{E} \mathbb{E}^{\top})\right) \cdot \nabla^{\ell} u d x \\
& =-\int_{\mathbb{R}^3}
\nonumber
\nabla^{\ell-1}
\left(\operatorname{div}(\mathbb{E} \mathbb{E}^{\top})\right)
\cdot \nabla^{\ell+1} u d x \\
& =-\sum_{0 \leq \kappa \leq \ell-1} C_{\ell-1}^\kappa \int_{\mathbb{R}^3}
\nonumber
\nabla^\kappa \mathbb{E} \nabla^{\ell-\kappa} \mathbb{E} \cdot \nabla^{\ell+1} u d x \\
\nonumber
& \lesssim \sum_{0 \leq \kappa \leq \ell-1}
\left\| \nabla^\kappa \mathbb{E} \nabla^{\ell-\kappa} \mathbb{E} \right\|_{L^2}\left\|
\nabla^{\ell+1} u\right\|_{L^2} \\
\nonumber
& \lesssim\left\|
\mathbb{E} \nabla^{\ell} \mathbb{E} \right\|_{L^2}
\left\|\nabla^{\ell+1} u\right\|_{L^2}
+\sum_{1 \leq \kappa \leq \ell-1}
\left\|\nabla^\kappa \mathbb{E} \nabla^{\ell-\kappa} \mathbb{E}\right\|_{L^2}
\left\|\nabla^{\ell+1} u\right\|_{L^2} \\
& \lesssim \delta\left\|\nabla^{\ell}(\mathbb{E}, \nabla u)\right\|_{L^2}^2
+\sum_{1 \leq \kappa \leq \ell-1}\left\|\nabla^\kappa \mathbb{E} \nabla^{\ell-\kappa} \mathbb{E}\right\|_{L^2}\left\|\nabla^{\ell+1} u\right\|_{L^2}.
\end{align}
Here, $\alpha$ serves to adjust the exponent between the energy and dissipation rates, and is defined by
\begin{align}\label{202512222035-alpha}
& \frac{\kappa}{3}=\left(\frac{\alpha}{3}-\frac{1}{2}\right) \nonumber \times\left(1-\frac{\kappa}{\ell}\right)+\left(\frac{\ell}{3}-\frac{1}{2}\right) \times \frac{\kappa}{\ell} \Rightarrow \\
& \alpha=\frac{3 \ell}{2(\ell-\kappa)} \in\left(\frac{3}{2}, 3\right], \quad \text { since } 1 \leq \kappa \leq \frac{\ell}{2} .
\end{align}
If $1 \leq \kappa \leq\left[\frac{\ell}{2}\right]$, by Lemma \ref{A1-N-202512221945}, then we get
\begin{align}\label{202512222037}
\| \nabla^\kappa \nonumber
& \mathbb{E} \nabla^{\ell-\kappa} \mathbb{E} \|_{L^2} \\ \nonumber
& \lesssim\left\|\nabla^\kappa \mathbb{E}\right\|_{L^{\infty}} \nonumber
\left\|\nabla^{\ell-\kappa} \mathbb{E}\right\|_{L^2} \\
& \lesssim\left\|\nabla^\alpha \mathbb{E}\right\|_{L^2}^{1-\frac{\kappa}{\ell}}
\left\|\nabla^{\ell} \mathbb{E}\right\|_{L^2}^{\frac{\kappa}{\ell}}
\left\|\mathbb{E}\right\|_{L^2}^{\frac{\kappa}{\ell}}
\left\|\nabla^{\ell} \mathbb{E}\right\|_{L^2}^{1-\frac{\kappa}{\ell}}
\lesssim \delta
\left\|\nabla^{\ell} \mathbb{E}\right\|_{L^2},
\end{align}
where $\alpha$ is the same one defined by \eqref{202512222035-alpha}. If $\left[\frac{\ell}{2}\right]+1 \leq \kappa \leq \ell-1$,
we then rewrite the factor to obtain

\begin{align}\label{202512222044-rewrite the factor}
\nonumber
& \left\|\nabla^\kappa \mathbb{E} \nabla^{\ell-\kappa} \mathbb{E}\right\|_{L^2} \\
\nonumber
& \quad \lesssim\left\|\nabla^\kappa \mathbb{E}\right\|_{L^2}
\left\|\nabla^{\ell-\kappa} \mathbb{E}\right\|_{L^{\infty}} \\
& \quad \lesssim
\left\|\mathbb{E}\right\|_{L^2}^{1-\frac{\varsigma}{\epsilon}}
\left\|\nabla^{\ell} \mathbb{E}\right\|_{L^2}^{\hbar}
\left\|\nabla^\alpha \mathbb{E}\right\|_{L^2}^{\star}
\left\|\nabla^{\ell} \mathbb{E}\right\|_{L^2}^{1-\frac{\varsigma}{\ell}}
\lesssim \delta\left\|\nabla^{\ell} \mathbb{E}\right\|_{L^2},
\end{align}
where $\alpha$ is defined by
\begin{align*}
\frac{\ell-k}{3} & =\left(\frac{\alpha}{3}-\frac{1}{2}\right) \times \frac{\kappa}{\ell}+\left(\frac{\ell}{3}-\frac{1}{2}\right) \times\left(1-\frac{\kappa}{\ell}\right) \Rightarrow \\
\alpha & =\frac{3 \ell}{2 \kappa} \leq 3, \quad \text { since } \kappa \geq \frac{\ell+1}{2} .
\end{align*}
By \eqref{202512221850-K1K2-K1}-\eqref{202512222044-rewrite the factor}, we deduce
\begin{align}\label{202512222051}
K_1 \lesssim \delta\left\|\nabla^{\ell}(\mathbb{E}, \nabla u)\right\|_{L^2}^2 .
\end{align}
Following a similar approach to that used for estimating $K_1$ , we get
\begin{align}\label{K2-202512222055}
K_2
=\int_{\mathbb{R}^3} \nabla^{\ell}(\nabla u \mathbb{E})
:
\nabla^{\ell} \mathbb{E} d x
\lesssim
\delta\left\|\nabla^{\ell}(\mathbb{E}, \nabla u)\right\|_{L^2}^2.
\end{align}
\begin{align}\label{K4-202512222121}
K_4:=-\int_{\mathbb{R}^3} \nabla^{\ell}\left(u \cdot \nabla u\right) \cdot \nabla^{\ell} u d x.
\end{align}
Then, for the term $K_3$, employing the commutator notation, we re-express it as
\begin{align}\label{K3-202512222058}
K_3
\nonumber
& :=-\int_{\mathbb{R}^3} \nabla^{\ell}(u \cdot \nabla \mathbb{E})
:
\nabla^{\ell} \mathbb{E} d x \\
& =-\int_{\mathbb{R}^3}\left(u \cdot \nabla \nabla^{\ell} \mathbb{E}
+\left[\nabla^{\ell}, u\right] \cdot \nabla \mathbb{E}\right)
:
\nabla^{\ell} \mathbb{E} d x:=K_{31}+K_{32}.
\end{align}
Integrating by parts yields
\begin{align}\label{integrating by part-K3-202512222102}
K_{31}
\nonumber
& :=-\int_{\mathbb{R}^3} u \cdot \nabla \frac{\left|\nabla^{\ell} \mathbb{E}\right|^2}{2} d x \\
& =\frac{1}{2} \int_{\mathbb{R}^3} \operatorname{div} u\left|\nabla^{\ell} \mathbb{E}\right|^2 d x \lesssim
\|\nabla u\|_{L^{\infty}}
\left\|\nabla^{\ell} \mathbb{E}\right\|_{L^2}^2
\lesssim \delta\left\|\nabla^{\ell} \mathbb{E}\right\|_{L^2}^2.
\end{align}
Applying the commutator estimate from Lemma \eqref{A7-commutator estimates-202512222029} and then Sobolev¡¯s inequality, we obtain the following bound:
\begin{align}\label{commutator estimate-K32-202512222105}
K_{32}
\nonumber
&
\lesssim\left(\|\nabla u\|_{L^{\infty}}
\left\|\nabla^{\ell-1} \nabla \mathbb{E}\right\|_{L^2}
+\left\|\nabla^{\ell} u\right\|_{L^6}
\|\nabla \mathbb{E}\|_{L^3}\right)\left\|
\nabla^{\ell} \mathbb{E}\right\|_{L^2} \\
& \lesssim \delta\left\|\nabla^{\ell}(\mathbb{E}, \nabla u)\right\|_{L^2}^2 .
\end{align}
Combining \eqref{K3-202512222058}¨C\eqref{commutator estimate-K32-202512222105}, we obtain
\begin{align}\label{K31K32-202512222109}
K_3 \lesssim \delta\left\|\nabla^{\ell}(\mathbb{E}, \nabla u)\right\|_{L^2}^2 .
\end{align}
\end{proof}
The third step consists in establishing dissipation estimates for $\mathbb{E}^T-\mathbb{E} $.
\begin{Lemma}\label{dissipation estimates-EE-202512222128}
Assume that
\begin{align*}
\|(u,\mathbb{E})\|_{H^3} \leq \delta \ll 1,
\end{align*}
and let the assumptions of Theorem \ref{pure-energy-202512221702} hold. Then, for any integer $0 \leq \ell \leq N-1$, we have
\begin{align}\label{dissipation estimates-EE-202512222129}
\nonumber
& \frac{d}{d t} \int_{\mathbb{R}^3}
\nabla^{\ell-1}\left(\nabla u-(\nabla u)^T\right) \cdot \nabla^{\ell+1}\left(\mathbb{E}^T-\mathbb{E}\right) d x
+C\left\|\nabla^{\ell+1}\left(\mathbb{E}^T-\mathbb{E}\right)\right\|_{L^2}^2 \\
& \quad
\lesssim
\left\|\nabla^{\ell+1}(u, \nabla u)\right\|_{L^2}^2
+\delta\left\|\nabla^{\ell+1}( \mathbb{E})\right\|_{L^2}^2.
\end{align}
\end{Lemma}
\begin{proof}
By \eqref{A2-A8-two important identities-202512222140} and $\mathbb{E}=\mathbb{F}-\mathbb{I}$, we get
\begin{align*}
\nabla_k \mathbb{E}^{i j}+\mathbb{E}^{l k} \nabla_l \mathbb{E}^{i j}=\nabla_j \mathbb{E}^{i k}+\mathbb{E}^{l j} \nabla_l \mathbb{E}^{i k}, \quad \forall t \geq 0 .
\end{align*}
Hence, it follows that
\begin{align}\label{curldiv-E-202512222146}
\operatorname{curl} \operatorname{div} \mathbb{E}
\nonumber
& =
\nabla_j \nabla_k \mathbb{E}^{i k}
-\nabla_i \nabla_k \mathbb{E}^{j k}
=\nabla_k \nabla_j \mathbb{E}^{i k}
-\nabla_k \nabla_i \mathbb{E}^{j k} \\
\nonumber
& =
\nabla_k \nabla_k \mathbb{E}^{i j}
-\nabla_k \nabla_k \mathbb{E}^{j i}
+\mathbb{R} \\
& =
\Delta\left(\mathbb{E}^{i j}-\mathbb{E}^{j i}\right)
+\mathbb{R}
=\Delta\left(\mathbb{E}-\mathbb{E}^T\right)
+\mathbb{R},
\end{align}
where the antisymmetric matrix $\mathbb{R}$ is defined as
\begin{align*}
\mathbb{R}^{i j}:=\nabla_k\left(\mathbb{E}^{l k} \nabla_l \mathbb{E}^{i j}
-\mathbb{E}^{l j} \nabla_l \mathbb{E}^{i k}\right)
-\nabla_k\left(\mathbb{E}^{l k} \nabla_l \mathbb{E}^{j i}
-\mathbb{E}^{l i} \nabla_l \mathbb{E}^{j k}\right) .
\end{align*}
Taking the difference between $\eqref{incompressible viscoelastic flow-E-202512201224}_2^T$ and $\eqref{incompressible viscoelastic flow-E-202512201224}_2$, we arrive at
\begin{align}\label{difference-E-ET-2025122231602}
\left(\mathbb{E}^T-\mathbb{E}\right)_t
+\nabla u-(\nabla u)^T
=(\nabla u \mathbb{E})^T
-\nabla u \mathbb{E}
-u \cdot \nabla\left(\mathbb{E}^T-\mathbb{E}\right) .
\end{align}
Applying $\nabla$ to Eq. $\eqref{incompressible viscoelastic flow-E-202512201224}_1$, by \eqref{curldiv-E-202512222146}, we find $\nabla u-(\nabla u)^T$ satisfying

\begin{align}\label{nabla u-nabla u-2025122231606}
\nonumber
& \left(\nabla u-(\nabla u)^T\right)_t
-\mu \Delta\left(\nabla u-(\nabla u)^T\right)
+\Delta\left(\mathbb{E}^T-\mathbb{E}\right) \\
& \quad=\nabla \mathbb{G}-(\nabla \mathbb{G})^T+\mathbb{R},
\end{align}
where
\begin{align*}
\mathbb{G}:
= \left(\mathbb{G}^1, \mathbb{G}^2, \mathbb{G}^3\right)^T,
\quad \mathbb{G}^i:=\mathbb{E}^{l k} \nabla_l \mathbb{E}^{i k} -u \cdot \nabla u^i
, \quad i=1,2,3.
\end{align*}
Let $0 \leq \ell \leq N-1$. Applying $\nabla^{\ell}$ to Eq. \eqref{nabla u-nabla u-2025122231606}, multiplying the resulting identity by $\nabla^{\ell}\left(\mathbb{E}^T-\mathbb{E}\right)$ and then integrating on $\mathbb{R}^3$, by integrating by parts, we have
We apply $\nabla^{\ell}$ to \eqref{nabla u-nabla u-2025122231606}, multiply the outcome by $\nabla^{\ell}\left(\mathbb{E}^T-\mathbb{E}\right)$, integrate over $\mathbb{R}^3$, and use integration by parts to obtain
\begin{align}\label{2025122231616-1}
\nonumber
\| \nabla^{\ell+1}
& \left(\mathbb{E}^T-\mathbb{E}\right) \|_{L^2}^2 \\
=
\nonumber
& \int_{\mathbb{R}^3}
\nabla^{\ell}
\left(\nabla u-(\nabla u)^T\right)_t \cdot
\nabla^{\ell}\left(\mathbb{E}^T-\mathbb{E}\right) d x \\
\nonumber
& -\mu \int_{\mathbb{R}^3}
\nabla^{\ell}
\Delta\left(\nabla u-(\nabla u)^T\right)
\cdot \nabla^{\ell}\left(\mathbb{E}^T-\mathbb{E}\right) d x \\
\nonumber
& -\int_{\mathbb{R}^3} \nabla^{\ell} \nabla_j
\left[\mathbb{E}^{l k} \nabla_l \mathbb{E}^{i k} \right.
 \left.-u \cdot \nabla u^{i }\right] \cdot \nabla^{\ell}\left(\mathbb{E}^{j i}-\mathbb{E}^{i j}\right) d x \\
\nonumber
& +\int_{\mathbb{R}^3} \nabla^{\ell}\left\{\nabla _ { j } \left[\mathbb{E}^{l k} \nabla_l \mathbb{E}^{i k}\right.\right.
 \left.\left.-u \cdot \nabla u^i\right]\right\}^T
\cdot \nabla^{\ell}\left(\mathbb{E}^{j i}-\mathbb{E}^{i j}\right) d x \\
\nonumber
& -\int_{\mathbb{R}^3} \nabla^{\ell} \mathbb{R}
\cdot \nabla^{\ell}\left(\mathbb{E}^T-\mathbb{E}\right) d x \\
=
\nonumber
& \frac{d}{d t} \int_{\mathbb{R}^3} \nabla^{\ell}
\left(\nabla u-\left(\nabla u\right)^T\right)
\cdot \nabla^{\ell}\left(\mathbb{E}^T-\mathbb{E}\right) d x \\
\nonumber
& -\int_{\mathbb{R}^3} \nabla^{\ell}
\left(\nabla u-(\nabla u)^T\right)
\cdot
\nabla^{\ell}\left(\mathbb{E}^T-\mathbb{E}\right)_t d x \\
\nonumber
& -\mu \int_{\mathbb{R}^3} \nabla^{\ell}
\Delta\left(\nabla u-(\nabla u)^T\right)
\cdot
\nabla^{\ell}\left(\mathbb{E}^T-\mathbb{E}\right) d x \\
\nonumber
& -\int_{\mathbb{R}^3} \nabla^{\ell}
\nabla_j\left(\mathbb{E}^{l k} \nabla_l \mathbb{E}^{i k}
-u \cdot \nabla u^i\right)
\cdot
\nabla^{\ell}\left(\mathbb{E}^{j i}-\mathbb{E}^{i j}\right) d x \\
\nonumber
& +\int_{\mathbb{R}^3} \nabla^{\ell}
\left[\nabla_j\left(\mathbb{E}^{l k} \nabla_l \mathbb{E}^{i k}
-u \cdot \nabla u^i\right)\right]^T
\cdot
\nabla^{\ell}\left(\mathbb{E}^{j i}-\mathbb{E}^{i j}\right) d x\\
\nonumber
& -\int_{\mathbb{R}^3} \nabla^{\ell}
\mathbb{R} \cdot
\nabla^{\ell}\left(\mathbb{E}^T-\mathbb{E}\right) d x \\
=
&
L_1+L_2+L_3+L_4+L_5+L_6.
\end{align}
Note that the structure of system \eqref{difference-E-ET-2025122231602}-\eqref{nabla u-nabla u-2025122231606} is analogous. First, for the term $L_1$, integrating by parts gives
\begin{align}\label{202512231646-L1}
L_1
\nonumber
& =\frac{d}{d t} \int_{\mathbb{R}^3} \nabla^{\ell}
\left(\nabla u-(\nabla u)^T\right)
\cdot \nabla^{\ell}\left(\mathbb{E}^T-\mathbb{E}\right) d x \\
& =-\frac{d}{d t} \int_{\mathbb{R}^3} \nabla^{\ell-1}
\left(\nabla u-(\nabla u)^T\right)
\cdot \nabla^{\ell+1}\left(\mathbb{E}^T-\mathbb{E}\right) d x.
\end{align}
Next, for the term $L_2$, using \eqref{difference-E-ET-2025122231602}, we obtain
\begin{align}\label{202512231648-L2}
L_2
\nonumber
= &
-\int_{\mathbb{R}^3} \nabla^{\ell}
\left(\nabla u-(\nabla u)^T\right)
\cdot \nabla^{\ell}\left(\mathbb{E}^T-\mathbb{E}\right)_t d x \\
\nonumber
= &
-\int_{\mathbb{R}^3} \nabla^{\ell}
\left(\nabla u-(\nabla u)^T\right)
\cdot \nabla^{\ell}
\left[(\nabla u \mathbb{E})^T-\nabla u \mathbb{E}\right. \\
\nonumber
& \left.-u \cdot \nabla\left(\mathbb{E}^T-\mathbb{E}\right)
-\left(\nabla u-(\nabla u)^T\right)\right] d x \\
= &
L_{21}+L_{22}+L_{23}+L_{24}.
\end{align}
With the estimates in \eqref{202512222037}-\eqref{202512222044-rewrite the factor} for $K_{1}$ at hand, together with $\mathbb{R}^{i j}=\nabla_k\left(\mathbb{E}^{l k} \nabla_l \mathbb{E}^{i j}-\right. \left.\mathbb{E}^{l j} \nabla_l \mathbb{E}^{i k}\right)-\nabla_k\left(\mathbb{E}^{l k} \nabla_l \mathbb{E}^{j i}-\mathbb{E}^{l i} \nabla_l \mathbb{E}^{j k}\right)$, we get the estimates of $L_{21}, L_{22}, L_{23}$, $L_4, L_5$ and $L_6$ :
\begin{align}\label{202512231716-L21}
L_{21}
\nonumber
& =-\int_{\mathbb{R}^3} \nabla^{\ell}
\left(\nabla u-(\nabla u)^T\right)
\cdot \nabla^{\ell}(\nabla u \mathbb{E})^T d x \\
& \lesssim
\sum_{0 \leq \kappa \leq \ell}
\left\|\nabla^\kappa \mathbb{E} \nabla^{\ell+1-\kappa} u\right\|_{L^2}
\left\|\nabla^{\ell+1} u\right\|_{L^2}
\lesssim
\delta\left\|\nabla^{\ell+1}(u, \mathbb{E})\right\|_{L^2}^2.
\end{align}
\begin{align}\label{202512231718-L22}
L_{22}
& =
-\int_{\mathbb{R}^3} \nabla^{\ell}
\left(\nabla u-(\nabla u)^T\right)
\cdot \nabla^{\ell}(-\nabla u \mathbb{E}) d x
\lesssim \delta\left\|\nabla^{\ell+1}(u, \mathbb{E})\right\|_{L^2}^2.
\end{align}
\begin{align}\label{202512231719-L23}
L_{23}
=-\int_{\mathbb{R}^3} \nabla^{\ell}
\left(\nabla u-(\nabla u)^T\right)
\cdot \nabla^{\ell}
\left[-u \cdot \nabla\left(\mathbb{E}^T-\mathbb{E}\right)\right] d x
\lesssim
\delta\left\|\nabla^{\ell+1}\left(u, \mathbb{E}^T-\mathbb{E}\right)\right\|_{L^2}^2.
\end{align}
\begin{align}\label{202512231722-L4}
L_4
\nonumber
& =-\int_{\mathbb{R}^3}
\nabla^{\ell}
\nabla_j\left(\mathbb{E}^{l k} \nabla_l \mathbb{E}^{i k}
-u \cdot \nabla u^i\right)
\cdot \nabla^{\ell}\left(\mathbb{E}^{j i}-\mathbb{E}^{i j}\right) d x \\
& \lesssim \delta
\left\|\nabla^{\ell+1}
\left(u, \mathbb{E}, \mathbb{E}^T-\mathbb{E}\right)\right\|_{L^2}^2.
\end{align}
\begin{align}\label{202512231725-L5}
L_5
\nonumber
& =\int_{\mathbb{R}^3} \nabla^{\ell}
\left[\nabla_j\left(\mathbb{E}^{l k} \nabla_l \mathbb{E}^{i k}
-u \cdot \nabla u^i\right)\right]^T
\cdot \nabla^{\ell}\left(\mathbb{E}^{j i}-\mathbb{E}^{i j}\right) d x \\
& \lesssim \delta
\left\|\nabla^{\ell+1}\left(u, \mathbb{E}, \mathbb{E}^T-\mathbb{E}\right)\right\|_{L^2}^2.
\end{align}
\begin{align}\label{202512231727-L6}
L_6
& =
-\int_{\mathbb{R}^3} \nabla^{\ell} \mathbb{R} \cdot \nabla^{\ell}
\left(\mathbb{E}^T-\mathbb{E}\right) d x
\lesssim \delta
\left\|\nabla^{\ell+1}\left(\mathbb{E}, \mathbb{E}^T-\mathbb{E}\right)\right\|_{L^2}^2.
\end{align}
Regarding the term $L_{24}$, it follows directly that
\begin{align}\label{202512231729-L24}
\left.L_{24}=-\int_{\mathbb{R}^3} \nabla^{\ell}
\left(\nabla u-(\nabla u)^T\right)
\cdot \nabla^{\ell}
\left(\nabla u-(\nabla u)^T\right)\right) d x
\lesssim\left\|\nabla^{\ell+1} u\right\|_{L^2}^2.
\end{align}
Concerning $L_3$, we get
\begin{align}\label{202512231730-L3}
L_3
\nonumber
& =-\mu \int_{\mathbb{R}^3} \nabla^{\ell} \Delta
\left(\nabla u-(\nabla u)^T\right)
\cdot \nabla^{\ell}
\left(\mathbb{E}^T-\mathbb{E}\right) d x \\
& \lesssim
\left\|\nabla^{\ell+2} u\right\|_{L^2}
\left\|\nabla^{\ell+1}\left(\mathbb{E}^T-\mathbb{E}\right)\right\|_{L^2}.
\end{align}
By putting \eqref{2025122231616-1} and \eqref{202512231730-L3} together, we get
\begin{align}\label{202512231734-1}
\| \nabla^{\ell+1}
\nonumber
& \left(\mathbb{E}^T-\mathbb{E}\right) \|_{L^2}^2 \\
\lesssim
\nonumber
& -\frac{d}{d t} \int_{\mathbb{R}^3} \nabla^{\ell-1}
\left(\nabla u-(\nabla u)^T\right) \cdot \nabla^{\ell+1}
\left(\mathbb{E}^T-\mathbb{E}\right) d x \\
\nonumber
& +\left\|\nabla^{\ell+1}
\left( \mathbb{E}^T-\mathbb{E}\right)\right\|_{L^2}
\left\|\nabla^{\ell+2} u\right\|_{L^2}
+\left\|\nabla^{\ell+1} u\right\|_{L^2}^2 \\
& +\delta\left\|\nabla^{\ell+1}\left(u, \mathbb{E},\left(\mathbb{E}^T-\mathbb{E}\right)\right)\right\|_{L^2}^2
+\delta\left\|\nabla^{\ell+2} u\right\|_{L^2}^2.
\end{align}
Therefore, by Cauchy's inequality and the assumption that $\delta$ is small, \eqref{dissipation estimates-EE-202512222129} follows from the preceding estimate.
\end{proof}
We conclude by recovering the estimates concerning $\mathbb{E}$.
\begin{Lemma}\label{recover the estimates on-E-202512231753}
Assume that
\begin{align*}
\|(u,\mathbb{E})\|_{H^3} \leq \delta \ll 1,
\end{align*}
and let the assumptions of Theorem \ref{pure-energy-202512221702} hold. Then, for any integer $0 \leq \ell \leq N-1$, we have
\begin{align}\label{recover the estimates on-E-2025122231754}
\left\|\nabla^{\ell+1} \mathbb{E}\right\|_{L^2}^2
\lesssim\left\|\nabla^{\ell+1}\left(\mathbb{E}^T-\mathbb{E}\right)\right\|_{L^2}^2 .
\end{align}
\end{Lemma}
\begin{proof}
By \eqref{curldiv-E-202512222146} and \eqref{A1-A8-two important identities-202512222139} $\operatorname{div}\left( \mathbb{F}^T\right)(t)=0$, we get
\begin{align}\label{202512231813-1}
\Delta \operatorname{div} \mathbb{E}
\nonumber
& =\nabla \operatorname{div} \operatorname{div} \mathbb{E}
-\text { curl curl } \operatorname{div} \mathbb{E} \\
\nonumber
& =\nabla \operatorname{div} \operatorname{div} \mathbb{E}^T
-\text { curl curl } \operatorname{div} \mathbb{E} \\
& =
-\operatorname{curl}\left(\Delta\left(\mathbb{E}-\mathbb{E}^T\right)\right)
-\operatorname{curl} \mathbb{R}.
\end{align}
Using the properties of the Riesz transforms (see, e.g., \cite{Grafakos2004}), we obtain
\begin{align}\label{202512231821-Riesz transforms}
\left\|\nabla^{\ell} \operatorname{div} \mathbb{E}\right\|_{L^2}
\lesssim
\left\|\nabla^{\ell+1}
\left(\mathbb{E}^T-\mathbb{E}\right)\right\|_{L^2}
+\left\|\nabla^{\ell}(\mathbb{E} \nabla \mathbb{E})\right\|_{L^2}.
\end{align}
By estimates analogous to the ones already established, we obtain
\begin{align}\label{202512231828-similar estimates}
\left\|\nabla^{\ell}(\mathbb{E} \nabla \mathbb{E})\right\|_{L^2}
\lesssim \delta\left\|\nabla^{\ell+1} \mathbb{E}\right\|_{L^2}.
\end{align}
From \eqref{A2-A8-two important identities-202512222140} and the relation $\mathbb{E}=\mathbb{F}-\mathbb{I}$, we obtain for all $t \geq 0$:
\begin{align*}
\nabla_k \mathbb{E}^{i j}+\mathbb{E}^{l k} \nabla_l \mathbb{E}^{i j}
=\nabla_j \mathbb{E}^{i k}+\mathbb{E}^{l j} \nabla_l \mathbb{E}^{i k},
\end{align*}
so that $\operatorname{curl} \mathbb{E}$ can be treated as a higher-order contribution.
Thus we get
\begin{align}\label{202512231834-curl-high order term}
\left\|\nabla^{\ell} \operatorname{curl} \mathbb{E}\right\|_{L^2}^2
\lesssim
\left\|\nabla^{\ell}(\mathbb{E} \nabla \mathbb{E})\right\|_{L^2}^2
\lesssim \delta\left\|\nabla^{\ell+1} \mathbb{E}\right\|_{L^2}^2 .
\end{align}
From \eqref{202512231821-Riesz transforms}-\eqref{202512231834-curl-high order term}, we obtain
\begin{align}\label{202512231837-Consequently}
\left\|\nabla^{\ell+1} \mathbb{E}\right\|_{L^2}^2
\nonumber
& \lesssim
\left\|\nabla^{\ell} \operatorname{div} \mathbb{E}\right\|_{L^2}^2
+\left\|\nabla^{\ell} \operatorname{curl} \mathbb{E}\right\|_{L^2}^2 \\
& \lesssim
\left\|\nabla^{\ell+1}\left(\mathbb{E}^T-\mathbb{E}\right)\right\|_{L^2}^2
+\delta\left\|\nabla^{\ell+1} \mathbb{E}\right\|_{L^2}^2.
\end{align}
From the above inequality and the smallness of $\delta$, we thus obtain \eqref{recover the estimates on-E-2025122231754}.
\end{proof}

\section{Negative Sobolev/Besov estimates}\label{202512231847-Negative Sobolev/Besov estimates-202512231853}
This section is devoted to the evolution of the solution to \eqref{incompressible viscoelastic flow-E-202512201224}-\eqref{initial-condition-202512201225} in negative Sobolev and Besov norms. For homogeneous Sobolev spaces, the analysis of nonlinear terms requires the regularity index $s$ to lie in $\bigl(0, \frac{3}{2}\bigr)$. Under this restriction, we establish the following lemma.
\begin{Lemma}\label{202512231856-Lemma 4.1}
Assume that
\begin{align*}
\|(u,\mathbb{E})\|_{H^3} \leq \delta \ll 1,
\end{align*}
and let the assumptions of Theorem \ref{pure-energy-202512221702} hold. Then, we get for $s \in\left(0, \frac{1}{2}\right]$,
\begin{align}\label{202512231856-4.1}
\nonumber
& \frac{d}{d t}\|(u,\mathbb{E})\|_{\dot{H}^{-s}}^2
+C\|\nabla u\|_{\dot{H}^{-s}}^2 \\
& \quad \lesssim\left(\|\nabla(u, \mathbb{E})\|_{H^1}^2\right)
\|(u,\mathbb{E})\|_{\dot{H}^{-s}} ;
\end{align}
and for $s \in\left(\frac{1}{2}, \frac{3}{2}\right)$,
\begin{align}\label{202512231856-4.1-2}
\nonumber
& \frac{d}{d t}\|(u,\mathbb{E})\|_{\dot{H}^{-s}}^2
+C\|\nabla u\|_{\dot{H}^{-s}}^2 \\
& \quad \lesssim
\left(\|\nabla(u, \mathbb{E})\|_{H^1}\right)^{\frac{5}{2}-s}
\|(u, \mathbb{E})\|_{L^2}^{s-\frac{1}{2}}
\|(u, \mathbb{E})\|_{\dot{H}^{-s}}.
\end{align}
\end{Lemma}

\begin{proof}
Applying $\Lambda^{-s}$ to the first two equations of \eqref{incompressible viscoelastic flow-E-202512201224}, multiplying by $\Lambda^{-s} u^i$ and $\Lambda^{-s} \mathbb{E}$ respectively, and integrating over $\mathbb{R}^3$ yields
\begin{align}\label{202512241409-4.1-1}
\nonumber
&
\frac{d}{d t}\|(u,\mathbb{E} )\|_{\dot{H}^{-s}}^2
+\mu\|\nabla u\|_{\dot{H}^{-s}}^2  \\
=
\nonumber
&
\Lambda^{-s}(\nabla u \mathbb{E}-u \cdot \nabla \mathbb{E}): \Lambda^{-s} \mathbb{E} d x \\
\nonumber
&
+\int_{\mathbb{R}^3} \Lambda^{-s}
\left[\mathbb{E}^{j k} \nabla_j \mathbb{E}^{i k}-u \cdot \nabla u^i\right.]
\cdot \Lambda^{-s} u^i d x \\
\lesssim
\nonumber
&
\|(\nabla u \mathbb{E}-u \cdot \nabla \mathbb{E})\|_{\dot{H}^{-s}}
\|\mathbb{E}\|_{\dot{H}^{-s}} \\
& +
\|[\mathbb{E}^{j k} \nabla_j \mathbb{E}^{i k}-u\cdot \nabla u^i
)]
\left\|_{\dot{H}^{-s}}\right\| u^i \|_{\dot{H}^{-s}} .
\end{align}
We estimate the right-hand side of \eqref{202512241409-4.1-1} as follows:
\begin{align}\label{202512241423-4.1-right-hand side-1}
\nonumber
& \|\nabla u \mathbb{E}-u \cdot \nabla \mathbb{E}\|_{\dot{H}^{-s}}
\lesssim
\|\nabla u \mathbb{E}-u \cdot \nabla \mathbb{E}\|_{L^{\frac{1}{2}+\frac{s}{3}}}
\lesssim\|\nabla(u, \mathbb{E})\|_{L^2}
\|(u, \mathbb{E})\|_{L^{\frac{3}{s}}} \\
&
\lesssim
\|\nabla(u, \mathbb{E})\|_{L^2}
\|\nabla(u,\mathbb{E})\|_{L^2}^{\frac{1}{2}+s}
\left\|\nabla^2(u, \mathbb{E})\right\|_{L^2}^{\frac{1}{2}-s}
\lesssim\|\nabla(u, \mathbb{E})\|_{H^1}^2 ;
\end{align}
\begin{align}\label{202512241423-4.1-right-hand side-2}
\nonumber
&
\| \mathbb{E}^{j k}  \nabla_j \mathbb{E}^{i k}-u \cdot \nabla u^i \|_{\dot{H}^{-s}} \\
\nonumber
&
\lesssim
\left\|\mathbb{E}^{j k} \nabla_j \mathbb{E}^{i k}-u \cdot \nabla u^i\right\|_{L^{\frac{1}{2}+\frac{s}{3}}} \\
\nonumber
&
\lesssim
\|\nabla(u, \mathbb{E})\|_{L^2}
\|(u,\mathbb{E})\|_{L^{\frac{3}{s}}} \\
\nonumber
&
\lesssim
\|\nabla(u, \mathbb{E})\|_{L^2}
\|\nabla(u, \mathbb{E})\|_{L^2}^{\frac{1}{2}+s}
\left\|\nabla^2(u, \mathbb{E})\right\|_{L^2}^{\frac{1}{2}-s} \\
&
\lesssim
\|\nabla(u, \mathbb{E})\|_{H^1}^2 .
\end{align}
Having established \eqref{202512231856-4.1} from \eqref{202512241409-4.1-1}-\eqref{202512241423-4.1-right-hand side-2}, we now turn to the case $s \in \bigl(\frac{1}{2}, \frac{3}{2}\bigr)$. Here, $\frac{1}{2}+\frac{s}{3}<1$ and $2 < \frac{3}{s} < 6$. Consequently, applying Holder's inequality, Sobolev's inequality, Young's inequality, and Lemma \ref{A2-HL-202512222020} to the right-hand side of \eqref{202512241409-4.1-1} yields the following estimate:
\begin{align}\label{202512241423-4.1-right-hand side-3}
\nonumber
&
\|\nabla u \mathbb{E}-u \cdot \nabla \mathbb{E}\|_{\dot{H}^{-s}} \\
\nonumber
&
\lesssim
\|\nabla u \mathbb{E}-u \cdot \nabla \mathbb{E}\|_{L^{\frac{1}{2}+\frac{s}{3}}}
\lesssim
\|\nabla(u, \mathbb{E})\|_{L^2}
\|(u, \mathbb{E})\|_{L^{\frac{3}{s}}} \\
\nonumber
&
\lesssim
\|\nabla(u, \mathbb{E})\|_{L^2}
\|(u, \mathbb{E})\|_{L^2}^{s-\frac{1}{2}}
\|\nabla(u, \mathbb{E})\|_{L^2}^{\frac{3}{2}-s} \\
&
\lesssim
\|(u, \mathbb{E})\|_{L^2}^{s-\frac{1}{2}}
\|\nabla(u, \mathbb{E})\|_{H^1}^{\frac{5}{2}-s} ;
\end{align}
\begin{align}\label{202512241423-4.1-right-hand side-4}
\nonumber
&
\| \mathbb{E}^{j k}  \nabla_j \mathbb{E}^{i k}-u \cdot \nabla u^i\|_{\dot{H}^{-s}} \\
\nonumber
&
\lesssim
\left\|\mathbb{E}^{j k} \nabla_j \mathbb{E}^{i k}-u \cdot \nabla u^i\right\|_{L^{\frac{1}{2}+\frac{3}{3}}} \\
\nonumber
&
\lesssim
\|\nabla(u, \mathbb{E})\|_{L^2}
\|(u, \mathbb{E})\|_{L^{\frac{3}{s}}} \\
\nonumber
&
\lesssim
\|\nabla(u, \mathbb{E})\|_{L^2}
\|(u, \mathbb{E})\|_{L^2}^{s-\frac{1}{2}}
\|\nabla(u, \mathbb{E})\|_{L^2}^{\frac{3}{2}-s} \\
&
\lesssim
\|(u, \mathbb{E})\|_{L^2}^{s-\frac{1}{2}}
\|\nabla(u, \mathbb{E})\|_{H^1}^{\frac{5}{2}-s} .
\end{align}
Combining \eqref{202512241409-4.1-1} with \eqref{202512241423-4.1-right-hand side-3}-\eqref{202512241423-4.1-right-hand side-4} yields \eqref{202512231856-4.1-2}, which completes the proof of Lemma \ref{202512231856-Lemma 4.1}.
\end{proof}
For homogeneous Besov spaces, the analysis of nonlinear terms requires the regularity index $s$ to lie in $\bigl(0, \frac{3}{2}\bigr]$. Under this restriction, we establish the following lemma.
\begin{Lemma}\label{202512241450-Lemma 4.2}
Assume that
\begin{align*}
\|(u,\mathbb{E})\|_{H^3} \leq \delta \ll 1,
\end{align*}
and let the assumptions of Theorem \ref{pure-energy-202512221702} hold. Then, we get for $s \in\left(0, \frac{1}{2}\right]$,
\begin{align}\label{202512241451-4.15}
\frac{d}{d t}\|(u,\mathbb{E})\|_{\dot{B}_{2, \infty}^{-s}}^2
+C\|\nabla u\|_{\dot{B}_{2, \infty}^{-s}}^2
 \lesssim
\left(\|\nabla(u, \mathbb{E})\|_{H^1}^2\right)
\|(u, \mathbb{E})\|_{\dot{B}_{2, \infty}^{-s}}^2,
\end{align}
and for $\left(\frac{1}{2}, \frac{3}{2}\right]$,
\begin{align}\label{202512241452-4.16}
\frac{d}{d t}\|(u, \mathbb{E})\|_{\dot{B}_{2, \infty}^{-s}}^2
+C\|\nabla u\|_{\dot{B}_{2, \infty}^{-s}}^2
\lesssim
\left(\|\nabla(u, \mathbb{E})\|_{H^1}\right)^{\frac{5}{2}-s}
\|(u, \mathbb{E})\|_{L^2}^{s-\frac{1}{2}}
\|(u,\mathbb{E})\|_{\dot{B}_{2, \infty}^{-s}}.
\end{align}
\end{Lemma}
\begin{proof}
Within the Besov framework, the $\dot{\Delta}_l$ energy estimate for \eqref{incompressible viscoelastic flow-E-202512201224}$_1$-$_2$, scaled by $2^{-2sl}$ and supremized over $l$, gives
\begin{align}\label{202512241501-1}
\nonumber
&
\frac{1}{2} \frac{d}{d t} \|  (u, \mathbb{E})\|_{\dot{B}_{2, \infty}^{-s}}^2
+\mu\| \nabla u\|_{\dot{B}_{2, \infty}^{-s}}^2\\
=
\nonumber
&
\sup _{l \in \mathbb{Z}} 2^{-2 s l}
\left(\int_{\mathbb{R}^3} \dot{\Delta}_l
(\nabla u \mathbb{E}-u \cdot \nabla \mathbb{E})
\cdot \dot{\Delta}_l \mathbb{E} d x\right) \\
\nonumber
&
+\sup _{l \in \mathbb{Z}} 2^{-2 s l}
\left(\int_{\mathbb{R}^3} \dot{\Delta}_l
\left(\mathbb{E}^{j k} \nabla_j \mathbb{E}^{i k}
-u \cdot \nabla u^i\right)
\cdot \dot{\Delta}_l u^i d x\right) \\
\lesssim
&
\|(\nabla u \mathbb{E}-u \cdot \nabla \mathbb{E})\|_{\dot{B}_{2, \infty}^{-s}}
\|\mathbb{E}\|_{\dot{B}_{2, \infty}^{-s}}
+\| [\mathbb{E}^{j k} \nabla_j \mathbb{E}^{i k}-u\cdot \nabla u^i)]
\left\|_{\dot{B}_{2, \infty}^{-s}}\right
\| u^i \|_{\dot{B}_{2, \infty}^{-s}}.
\end{align}
The proof follows closely that of Lemma \ref{202512231856-Lemma 4.1}, with the key modification that we now apply Lemma \ref{A3-B-202512222024} to estimate the $\dot{B}_{2,\infty}^{-s}$ norm. Since $s=\frac{3}{2}$ is allowed here, we obtain estimates \eqref{202512241451-4.15}-\eqref{202512241452-4.16}, thereby completing the proof of Lemma \ref{202512241450-Lemma 4.2}.
\end{proof}
\section{Proof of Theorem \ref{pure-energy-202512221702}
\eqref{202512221715-globalsolution-incompressible viscoelastic flows}-\eqref{pure-energy-202512221703}
} \label{202512241526-}
This section is devoted to proving the global existence and uniqueness of solutions to the Cauchy problem \eqref{incompressible viscoelastic flow-E-202512201224}-\eqref{initial-condition-202512201225} (equivalently, \eqref{incompressible viscoelastic flow-202512201217}-\eqref{initial-condition-202512201220}). For completeness, we first recall the local well-posedness result, whose proof follows from standard arguments (see, e.g., \cite{Matsumura-Nishida1980}) and is therefore omitted.
\begin{Lemma}\label{202512241532-Lemma-admits a unique global solution}
Let the initial data satisfy $\left(u_0, \mathbb{E}_0\right) \in H^3\left(\mathbb{R}^3\right)$. Then there exists a constant $T>0$ such that the Cauchy problem \eqref{incompressible viscoelastic flow-E-202512201224}-\eqref{initial-condition-202512201225} possesses a unique solution $(u,\mathbb{E} ) $ with
\begin{align*}
\mathbb{E} \in C^0\left(0, T ; H^3\left(\mathbb{R}^3\right)\right) \cap C^1\left(0, T ; H^2\left(\mathbb{R}^3\right)\right), \\
u \in C^0\left(0, T ; H^3\left(\mathbb{R}^3\right)\right) \cap C^1\left(0, T ; H^1\left(\mathbb{R}^3\right)\right) .
\end{align*}
In addition, the following holds
\begin{align*}
\sup _{0 \leq t \leq T}\|(u,\mathbb{E})(t)\|_{H^3}
\lesssim
\left\|\left(u_0,\mathbb{E}_0\right)\right\|_{H^3}.
\end{align*}
\end{Lemma}
\begin{proof}
We begin by closing the energy estimates from Section \ref{Energy Estimates-202512221729} at each level $\ell$. Let $N \geq 3$, and choose integers $\kappa_1, \kappa_2$ such that $0 \leq \kappa_1 \leq \kappa_2-1$ and $3 \leq \kappa_2 \leq N$. Provided $\delta>0$ is sufficiently small, we now sum three families of estimates: (i) \eqref{one type of energy estimates including-202512221740} of Lemma \ref{one type of energy estimates including-202512221737} for $\ell=\kappa_1$ to $\kappa_2-1$; (ii) \eqref{another type of energy estimates-202512221839} of Lemma \ref{another type of energy estimates-202512221839-202512231850} for $\ell=\kappa_1+1$ to $\kappa_2$; and (iii) \eqref{dissipation estimates-EE-202512222129} of Lemma \ref{dissipation estimates-EE-202512222128} for $\ell=\kappa_1$ to $\kappa_2-1$. The result is
\begin{align}\label{202512241620-5-1}
\sum_{\ell=\kappa_1}^{\kappa_2-1}
\left(\frac{d}{d t}\left\|\nabla^{\ell}(u, \mathbb{E})\right\|_{L^2}^2
+C\left\|\nabla^{\ell+1} u\right\|_{L^2}^2\right)
 \lesssim
\sum_{\ell=\kappa_1}^{\kappa_2-1} \delta
\left\|\nabla^{\ell+1}(u, \mathbb{E})\right\|_{L^2}^2 ;
\end{align}
\begin{align}\label{202512241621-5-2}
\sum_{\ell=\kappa_1+1}^{\kappa_2}
\left(\frac{d}{d t}
\left\|\nabla^{\ell}(u, \mathbb{E})\right\|_{L^2}^2
+C\left\|\nabla^{\ell+1} u\right\|_{L^2}^2\right)
\lesssim
\sum_{\ell=\kappa_1+1}^{\kappa_2} \delta
\left\|\nabla^{\ell}(u, \mathbb{E})\right\|_{L^2}^2 ;
\end{align}
\begin{align}\label{202512241622-5-4}
\nonumber
&
\sum_{\ell=\kappa_1}^{\kappa_2-1}
\left(\frac{d}{d t} \int_{\mathbb{R}^3} \nabla^{\ell-1}
\left(\nabla u-(\nabla u)^T\right)
\cdot \nabla^{\ell+1}
\left(\mathbb{E}^T-\mathbb{E}\right) d x
+C\left\|\nabla^{\ell+1}\left(\mathbb{E}^T-\mathbb{E}\right)\right\|_{L^2}^2\right) \\
& \quad
\lesssim
\sum_{\ell=\kappa_1}^{\kappa_2-1}
\left(\left\|\nabla^{\ell+1}(u, \nabla u)\right\|_{L^2}^2
+\delta\left\|\nabla^{\ell+1}(u, \mathbb{E})\right\|_{L^2}^2\right).
\end{align}
Since $\delta$ is small, combining \eqref{202512241620-5-1}-\eqref{202512241621-5-2} with $\delta$ times \eqref{202512241622-5-4} yields
\begin{align*}
&
\frac{d}{d t}\left\{\sum_{\ell=\kappa_1}^{\kappa_2}
\left\|\nabla^{\ell}(u,\mathbb{E})\right\|_{L^2}^2
\right. \\
& \left.\left.
\quad
+\int_{\mathbb{R}^3} \nabla^{\ell-1}
\left(\nabla u-(\nabla u)^T\right)
\cdot
\nabla^{\ell+1}
\left(\mathbb{E}^T-\mathbb{E}\right) d x\right]\right\} \\
&
\quad
+C\left[\sum_{\ell=\kappa_1+1}^{\kappa_2+1}
\left\|\nabla^{\ell} u\right\|_{L^2}^2
+\sum_{\ell=\kappa_1+1}^{\kappa_2}
\left\|\nabla^{\ell}
\left(\mathbb{E}^T-\mathbb{E}\right)\right\|_{L^2}^2\right]
 \\
&
\lesssim
\sum_{\ell=\kappa_1+1}^{\kappa_2+1} \delta
\left\|\nabla^{\ell} u\right\|_{L^2}^2
+\sum_{\ell=\kappa_1+1}^{\kappa_2} \delta
\left\|\nabla^{\ell} \mathbb{E}\right\|_{L^2}^2
 \\
&
=
\sum_{\ell=\kappa_1+1}^{\kappa_2+1} \delta
\left\|\nabla^{\ell} u\right\|_{L^2}^2
+\sum_{\ell=\kappa_1+1}^{\kappa_2} \delta
\left\|\nabla^{\ell}(\mathbb{E})\right\|_{L^2}^2.
\end{align*}
Therefore, from \eqref{recover the estimates on-E-2025122231754} of Lemma \ref{recover the estimates on-E-202512231753}, we deduce
\begin{align}\label{202512241717-2}
\nonumber
&
\frac{d}{d t}\left\{\sum_{\ell=\kappa_1}^{\kappa_2}
\left\|\nabla^{\ell}(u, \mathbb{E})\right\|_{L^2}^2
+\sum_{\ell=\kappa_1}^{\kappa_2-1} \delta
[
\int_{\mathbb{R}^3} \nabla^{\ell-1}
\left(\nabla u-(\nabla u)^T\right) \cdot \nabla^{\ell+1}
\left(\mathbb{E}^T-\mathbb{E}\right) d x ]
\right\} \\
&
\quad
+C\left[\sum_{\ell=\kappa_1+1}^{\kappa_2+1}
\left\|\nabla^{\ell} u\right\|_{L^2}^2
+\sum_{\ell=\kappa_1+1}^{\kappa_2}\left\|\nabla^{\ell}
\left(\mathbb{E}^T-\mathbb{E}\right)\right\|_{L^2}^2\right] \leq 0.
\end{align}
For simplicity, we introduce the temporal energy functional, defined for $t \geq 0$ by
\begin{align}\label{202512241719-define the temporal energy functional}
D_{\kappa_1}^{\kappa_2}(t):
=
\sum_{\ell=\kappa_1}^{\kappa_2} C^{-1}
\left\|\nabla^{\ell}(u, \mathbb{E})\right\|_{L^2}^2
+\sum_{\ell=\kappa_1}^{\kappa_2-1} \delta C^{-1}
\times
\left[
\int_{\mathbb{R}^3} \nabla^{\ell-1}
\left(\nabla u-(\nabla u)^T\right)
\cdot \nabla^{\ell+1}
\left(\mathbb{E}^T-\mathbb{E}\right) d x\right] .
\end{align}
From Lemma \ref{recover the estimates on-E-202512231753}, we have that for $0 \leq \kappa_1 \leq \kappa_2-1$, the functional $D_{\kappa_1}^{\kappa_2}(t)$ is equivalent to $\left\|\nabla^{\kappa_1}(u, \mathbb{E})\right\|_{H^{\kappa_2-\kappa_1}}^2$. Together with estimate \eqref{recover the estimates on-E-2025122231754} from the same lemma, this yields
\begin{align}\label{202512241729-1}
\frac{d}{d t} D_{\kappa_1}^{\kappa_2}(t)
+
\left\|\nabla^{\kappa_1+1} u\right\|_{H^{\kappa_2-\kappa_1}}^2
+
\left\|\nabla^{\kappa_1+1} \mathbb{E}\right\|_{H^{\kappa_2-\kappa_1-1}}^2 \leq 0.
\end{align}
Setting $\kappa_1=0$ and $\kappa_2=3$ in the above inequality and integrating in time yields
\begin{align}\label{202512241729-1-1732-integrating it directly in time}
D_0^3(t)+\int_0^t
\left(\|\nabla u(\tau)\|_{H^3}^2
+\|\nabla \mathbb{E}(\tau)\|_{H^2}^2\right) d \tau \lesssim D_0^3(0) .
\end{align}
This closes the a priori estimate $|(u,\mathbb{E})(t)|_{H^3} \le \delta$ for small initial data. By a continuity argument, the local solution from Lemma \ref{202512241532-Lemma-admits a unique global solution} extends uniquely to a global one, proving Theorem \ref{pure-energy-202512221702}. Moreover, taking $\kappa_1=0$ and $3 \le \kappa_2 \le N$ in \eqref{202512241729-1} and integrating in time gives the energy inequality \eqref{202512221715-globalsolution-incompressible viscoelastic flows} stated in the theorem.

We now turn to proving the optimal decay rates \eqref{pure-energy-202512221703} in Theorem \ref{pure-energy-202512221702}. Under the hypotheses of the theorem, we first assert that either inequality \eqref{202512241739-1-5} or \eqref{202512241740-1-6} holds. An application of Lemma \ref{A4-special Sobolev interpolation-202512222025} or Lemma \ref{A5-special Besov interpolation-202512222027} then yields
\begin{align*}
\left\|\nabla^{\ell} f\right\|_{L^2} \leq\|f\|_{\dot{H}^{-s}}^{\frac{1+1+s}{2+s}}\left\|\nabla^{\ell+1} f\right\|_{L^2}^{\frac{\ell+s}{2+1-s}} \lesssim\left\|\nabla^{\ell+1} f\right\|_{L^2}^{\frac{\ell+s}{\ell+1+s}}
\end{align*}
or
\begin{align*}
\left\|\nabla^{\ell} f\right\|_{L^2} \leq\|f\|_{\dot{B}_{2, \infty}^{-s}}^{\frac{\ell+1+s}{-s}}\left\|\nabla^{\ell+1} f\right\|_{L^2}^{\frac{\ell+s}{1+s}} \lesssim\left\|\nabla^{\ell+1} f\right\|_{L^2}^{\frac{\ell+s}{\ell+1+s}},
\end{align*}
for $\ell=0,1, \ldots, N-1$, respectively. Combined with \eqref{202512221715-globalsolution-incompressible viscoelastic flows}, this implies that
\begin{align}\label{202512241748-1-4-implies that}
\left\|\nabla^{\ell}(u, \mathbb{E})\right\|_{H^{N-\ell}}^{2\left(1+\frac{1}{\ell+s}\right)}
\lesssim
\left\|\nabla^{\ell+1}(u, \mathbb{E})\right\|_{H^{N-\ell-1}}^2,
\end{align}
for $\ell=0,1, \ldots, N-1$.

Substituting \eqref{202512241748-1-4-implies that} into \eqref{202512241729-1} yields
\begin{align*}
\frac{d}{d t} D_{\ell}^N(t)+C\left[D_{\ell}^N(t)\right]^{1+\frac{1}{\ell+s}} \leq 0.
\end{align*}
Direct integration of this differential inequality yields
\begin{align*}
D_{\ell}^N(t) \leq C_0(1+t)^{-(\ell+s)},
\end{align*}
for $\ell=0,1, \ldots, N-1$.
Thus we prove the decay rates \eqref{pure-energy-202512221703}.

It remains to establish the auxiliary estimates \eqref{202512241739-1-5} and \eqref{202512241740-1-6} required in Theorem \ref{pure-energy-202512221702}. For the range $s \in \bigl(0, \frac{1}{2}\bigr]$, a direct time integration of \eqref{202512231856-4.1} gives
\begin{align*}
&
\|(u, \mathbb{E})(t)\|_{\dot{H}^{-s}}^2 \\
&
\quad
\lesssim
\left\|\left(u_0, \mathbb{E}_0\right)\right\|_{\dot{H}^{-s}}^2
+\int_0^t
\left(\|\nabla(u, \mathbb{E})\|_{H^1}^2\right)
\|(u, \mathbb{E})(\tau)\|_{\dot{H}^{-s}} d \tau \\
&
\quad
\leq C_0+C_0 \sup _{0 \leq \tau \leq t}
\|(u, \mathbb{E})(t)\|_{\dot{H}^{-s}}.
\end{align*}
Thus, \eqref{202512241739-1-5} holds for $s \in (0, \frac12]$. It remains to prove it for $s \in (\frac12, \frac32)$. Using the embedding $\dot{H}^{-s} \cap L^2 \subset \dot{H}^{-s'}$ ($0 \le s' \le s$), we obtain for $s=\frac12$ the decay estimates:
\begin{align}\label{202512241804-5-8}
D_{\ell}^N(t) \leq C_0(1+t)^{-\left(\ell+\frac{1}{2}\right)}, \quad \ell=0,1, \ldots, N-1.
\end{align}
Combining \eqref{202512231856-4.1-2} and \eqref{202512241804-5-8} gives, for $s \in \bigl(\frac{1}{2}, \frac{3}{2}\bigr)$,
\begin{align*}
&
\|(u, \mathbb{E})(t)\|_{\dot{H}^{-s}}^2
\\
&
\lesssim
\left\|\left(u_0, \mathbb{E}_0\right)\right\|_{\dot{H}^{-s}}^2
+\int_0^t
\left(\|\nabla(u, \mathbb{E})\|_{H^1}\right)^{\frac{5}{2}-s}
\|(u, \mathbb{E})\|_{L^2}^{s-\frac{1}{2}}
\|(u,\mathbb{E})(\tau)\|_{\dot{H}^{-s}} d \tau \\
&
\leq
C_0+C_0 \int_0^t(1+\tau)^{-\frac{3}{4}\left(\frac{5}{2}-s\right)}
\cdot
(1+\tau)^{-\frac{1}{4}\left(s-\frac{1}{2}\right)} d \tau
\cdot \sup _{0 \leq \tau \leq t}\|(u, \mathbb{E})(\tau)\|_{\dot{H}^{-s}} \\
&
\leq
C_0+C_0 \sup _{0 \leq \tau \leq t}\|(u, \mathbb{E})(\tau)\|_{\dot{H}^{-s}}.
\end{align*}
Thus, \eqref{202512241739-1-5} holds for $s \in \bigl(\frac{1}{2}, \frac{3}{2}\bigr)$. Estimate \eqref{202512241740-1-6} follows analogously. This completes the proof of Theorem \ref{pure-energy-202512221702} (namely, \eqref{202512221715-globalsolution-incompressible viscoelastic flows} and \eqref{pure-energy-202512221703}).
\end{proof}

\section{Proof of Theorem \ref{pure-energy-202512221702}
\eqref{N-pure-energy-202512221703}}\label{202512241814-Theorem-N-pure-energy-202512221703}

This section focuses on proving the optimal decay rate for the $N$-th order spatial derivatives of the global small solution to the Cauchy problem \eqref{incompressible viscoelastic flow-E-202512201224}-\eqref{initial-condition-202512201225} (equivalently, \eqref{incompressible viscoelastic flow-202512201217}-\eqref{initial-condition-202512201220}). A key ingredient is the following lemma, which establishes the time integrability of these $N$-th order derivatives.
\begin{Lemma}\label{dissipative}
Under the assumptions of Theorem \ref{pure-energy-202512221702}, for any arbitrarily fixed constant $0<\epsilon_0<1$, it holds that
\begin{align}\label{dissipative1}
&
\quad
(1+t)^{N+s-1}
\left\|\nabla^{N-1}(u,\mathbb{E})\right\|_{H^1}^2 \\ \nonumber
&
+(1+t)^{-\epsilon_0} \int_0^t(1+\tau)^{N+s+\epsilon_0-1}
\left(\left\|\nabla^{N} u \right\|_{H^1}^2
+\left\|\nabla^{N} \mathbb{E} \right\|_{L^2}^2\right) d \tau \leq C,
\end{align}
for some positive constant $C$ independent of $t$.
\end{Lemma}
\begin{proof}
Combining the results of Lemmas \ref{one type of energy estimates including-202512221737}, \ref{another type of energy estimates-202512221839-202512231850}, \ref{dissipation estimates-EE-202512222128}, and \ref{recover the estimates on-E-202512231753} yields the following estimates for every $0 \leq \ell \leq N-1$:
\begin{align}\label{Tanenergy1}
\frac{d}{d t}
\left\|\nabla^{\ell} (u,\mathbb{E}) \right\|_{L^2}^2
+C\left\|\nabla^{\ell+1} u \right\|_{L^2}^2
\lesssim
\delta
\left\|\nabla^{\ell+1}(u,\mathbb{E})\right\|_{L^2}^2,
\end{align}
\begin{align}\label{Tanenergy2}
\frac{d}{d t}
\left\|\nabla^{\ell+1}(u,\mathbb{E})\right\|_{L^2}^2
+C\left\|\nabla^{\ell+2} u \right\|_{L^2}^2
\lesssim
\delta
\left\|\nabla^{\ell+1}(u, \mathbb{E})\right\|_{L^2}^2,
\end{align}
\begin{align}\label{Tanenergy4}
\nonumber
&
\frac{d}{d t} \int_{\mathbb{R}^3} \nabla^{\ell-1}
\left(\nabla u-(\nabla u)^T\right) \cdot \nabla^{\ell+1}
\left(\mathbb{E}^T-\mathbb{E}\right) d x
+C\left\|\nabla^{\ell+1}
\left(\mathbb{E}^T-\mathbb{E}\right)\right\|_{L^2}^2\\
&
\quad
\lesssim
\left\|\nabla^{\ell+1}(u, \nabla u)\right\|_{L^2}^2
+\delta\left\|\nabla^{\ell+1}(\mathbb{E})\right\|_{L^2}^2,
\end{align}
\begin{align}\label{Tanenergy5}
\left\|\nabla^{\ell+1} \mathbb{E} \right\|_{L^2}^2
\lesssim
\left\|\nabla^{\ell+1}
\left(\mathbb{E}^T-\mathbb{E}\right)\right\|_{L^2}^2,
\end{align}
where the constant $C$ is a positive constant independent of time.
From estimates \eqref{Tanenergy1}, \eqref{Tanenergy2}, \eqref{Tanenergy4}, and \eqref{Tanenergy5}, it follows that for all $0 \leq k \leq N-1$,
\begin{align}\label{energy1}
\nonumber
&
\frac{d}{d t}
\left(\left\|\nabla^\ell(u,\mathbb{E})\right\|_{H^1}^2
+\eta_2 \int_{\mathbb{R}^3} \nabla^{\ell-1}
\left(\nabla u-(\nabla u)^T\right) \cdot \nabla^{\ell+1}
\left(\mathbb{E}^T-\mathbb{E}\right) d x\right)\\
&
\quad
+C\left\|\nabla^{\ell+1} u \right\|_{H^1}^2
+\eta_3\left\|\nabla^{\ell+1} \mathbb{E} \right\|_{L^2}^2
\leq 0.
\end{align}
Here $\eta_2,\eta_3$ is a small positive constant.
Setting $\ell = N-1$ in \eqref{energy1} immediately gives
\begin{align}\label{efe}
\frac{d}{d t} \mathcal{E}^{N-1}(t)
+C\left\|\nabla^{N} u \right\|_{H^1}^2
+\eta_3 \left\|\nabla^{N} \mathbb{E} \right\|_{L^2}^2 \leq 0,
\end{align}
where the energy $\mathcal{E}^{N-1}(t)$ is defined by
\begin{align*}
\mathcal{E}^{N-1}(t) \stackrel{\text { def }}{=}
\left\|\nabla^{N-1}(u,\mathbb{E})\right\|_{H^1}^2
+\eta_2 \int_{\mathbb{R}^3} \nabla^{N-2}
\left(\nabla u-(\nabla u)^T\right) \cdot \nabla^{N}
\left(\mathbb{E}^T-\mathbb{E}\right) d x.
\end{align*}
The smallness of $\eta_2$ guarantees the existence of two time-independent constants $c_1$ and $c_2$ such that
\begin{align}\label{equi}
c_1\left\|\nabla^{N-1}(u,\mathbb{E})\right\|_{H^1}^2 \leq \mathcal{E}^{N-1}(t)
\leq
c_2\left\|\nabla^{N-1}(u,\mathbb{E})\right\|_{H^1}^2.
\end{align}
Multiplying inequality \eqref{efe} by $(1+t)^{N+s+\epsilon_0-1}$, with $\epsilon_0 \in (0,1)$ fixed, we obtain
\begin{align}\label{multiplying}
\nonumber
&
\frac{d}{d t}\left\{(1+t)^{N+s+\epsilon_0-1} \mathcal{E}^{N-1}(t)\right\}
+(1+t)^{N+s+\epsilon_0-1}
\left(\left\|\nabla^{N} u \right\|_{H^1}^2
+\left\|\nabla^{N} \mathbb{E} \right\|_{L^2}^2\right) \\
 &\quad \leq C(1+t)^{N+s+\epsilon_0-2} \mathcal{E}^{N-1}(t).
\end{align}
Using the decay estimate \eqref{pure-energy-202512221703} and the equivalence \eqref{equi}, it is straightforward to verify that for $0 < k \leq N-1$,
\begin{align*}
(1+t)^{N+s+\epsilon_0-2} \mathcal{E}^{N-1}(t)
\leq C(1+t)^{N+s+\epsilon_0-2}\left\|\nabla^{N-1}(u,\mathbb{E})\right\|_{H^1}^2
\leq C(1+t)^{-1+\epsilon_0}.
\end{align*}
Thus, the above estimate together with \eqref{multiplying} gives
\begin{align}\label{multiplying2}
 \frac{d}{d t}\left\{(1+t)^{N+s+\epsilon_0-1} \mathcal{E}^{N-1}(t)\right\}
+(1+t)^{N+s+\epsilon_0-1}
\left(\left\|\nabla^{N} u \right\|_{H^1}^2
+\left\|\nabla^{N} \mathbb{E} \right\|_{L^2}^2\right)
\leq  C(1+t)^{-1+\epsilon_0}.
\end{align}
Integration of \eqref{multiplying2} over $[0, t]$ yields
\begin{align}\label{Integrating}
\nonumber
&
(1+t)^{N+s+\epsilon_0-1} \mathcal{E}^{N-1}(t)
+\int_0^t(1+\tau)^{N+s+\epsilon_0-1}
\left(\left\|\nabla^{N} \widetilde{u}\right\|_{H^1}^2
+\left\|\nabla^{N} \mathbb{E} \right\|_{L^2}^2\right) d \tau \\
&
\quad
\leq
\mathcal{E}^{N-1}(0)
+\int_0^t(1+\tau)^{-1+\epsilon_0} d \tau
\leq C(1+t)^{\epsilon_0},
\end{align}
which, combined with the equivalence \eqref{equi}, directly yields
\begin{align}\label{Integratingtogether}
&
\quad(1+t)^{N+s+\epsilon_0-1}
\left\|\nabla^{N-1}(u,\mathbb{E})\right\|_{H^1}^2 \\
\nonumber
&
+\int_0^t(1+\tau)^{N+s+\epsilon_0-1}
\left(\left\|\nabla^{N} u \right\|_{H^1}^2
+\left\|\nabla^{N} \mathbb{E}\right\|_{L^2}^2\right)d \tau
\leq C(1+t)^{\epsilon_0}.
\end{align}
Hence, we have \eqref{dissipative1}.
\end{proof}
Finally, using the time integrability of the dissipative term for $\nabla^N (u,\mathbb{E})$ provided by Lemma \ref{dissipative}, we establish its optimal decay rate of $N-t h$ one.
\begin{Lemma}\label{dissipativecan}
The global solution $(u,\mathbb{E})$ obeys the decay estimate, provided the assumptions of Theorem \ref{pure-energy-202512221702} hold:
\begin{align}\label{dissipative1can}
(1+t)^{N+s}\left\|\nabla^N(u,\mathbb{E})\right\|_{L^2}^2
+(1+t)^{-\epsilon_0} \int_0^t(1+\tau)^{N+s+\epsilon_0}\left\|\nabla^{N+1} u\right\|_{L^2}^2 d \tau \leq C,
\end{align}
for some positive constant $C$ independent of $t$.
\end{Lemma}
\begin{proof}
We apply $\nabla^N$ to \eqref{incompressible viscoelastic flow-E-202512201224}, multiply the resulting equation by $\nabla^N(u,\mathbb{E})$, and take the $L^2$ norm to obtain

\begin{align}\label{do}
&
\quad
\frac{d}{d t}\left\|\nabla^N(u,\mathbb{E})\right\|_{L^2}^2
+\left\|\nabla^{N+1} u \right\|_{L^2}^2 \\
\nonumber
&
\leq
C \int \nabla^N g_1 \cdot \nabla^N u d x
+C \int \nabla^N g_2 \cdot \nabla^N \mathbb{E} d x.
\end{align}
We now estimate the two terms on the right-hand side of \eqref{do}.
Integration by parts followed by Holder's inequality gives
\begin{align}\label{235}
 \int \nabla^N g_1 \cdot \nabla^N u d x
\nonumber
&
\leq
C\left\|\nabla^{N-1} (g_i)\right\|_{L^2}
\left\|\nabla^{N+1} u\right\|_{L^2} \\
\nonumber
&
\leq
C\left(\left\|\nabla^{N-1}(u \cdot \nabla u)\right\|_{L^2}
\right. \\
&\quad
+\int \nabla^N (\mathbb{E} \nabla \mathbb{E}) \cdot \nabla^N u d x\\
\nonumber
&
\stackrel{\text { def }}{=}
H_1+H_4  .
\end{align}
An application of the Sobolev inequality, Holder's inequality, and Lemma \ref{A7-commutator estimates-202512222029} yields
\begin{align}\label{236}
H_1
\nonumber
&
\leq
C\left(\|u\|_{L^{\infty}}\left\|\nabla^N u\right\|_{L^2}
+\left\|\left[\nabla^{N-1}, u\right] \cdot \nabla u\right\|_{L^2}\right)
\left\|\nabla^{N+1} u\right\|_{L^2} \\
\nonumber
&
\leq C\left(\|u\|_{L^{\infty}}\left\|\nabla^N u\right\|_{L^2}
+\|\nabla u\|_{L^{\infty}}\left\|\nabla^{N-1} u\right\|_{L^2}\right)
\left\|\nabla^{N+1} u\right\|_{L^2} \\
\nonumber
&
\leq
\epsilon\left\|\nabla^{N+1} u\right\|_{L^2}^2
+C_\epsilon\|\nabla u\|_{H^1}^2
\left\|\nabla^N u\right\|_{L^2}^2
+C_\epsilon\left\|\nabla^2 u\right\|_{H^1}^2
\left\|\nabla^{N-1} u\right\|_{L^2}^2 \\
&
\leq
\epsilon\left\|\nabla^{N+1} u\right\|_{L^2}^2
+C_\epsilon(1+t)^{-(1+s)}\left\|\nabla^N u\right\|_{L^2}^2
+C_\epsilon(1+t)^{-(N+1+2 s)},
\end{align}
where the last step uses the decay estimate \eqref{pure-energy-202512221703}.
Similarly, applying Lemma \ref{A7-commutator estimates-202512222029} and the decay estimate \eqref{pure-energy-202512221703} once more yields
\begin{align}\label{2312}
H_4
&
\leq
\epsilon\left\|\nabla^{N+1} u\right\|_{L^2}^2
+C_\epsilon(1+t)^{-1}
\left\|\nabla^N(u,\mathbb{E})\right\|_{L^2}^2.
\end{align}
Substituting estimates \eqref{236} and \eqref{2312} into \eqref{235} and applying Young's inequality yields
\begin{align}\label{2313}
\int \nabla^N g_i \cdot \nabla^N u d x
\leq
\left(\epsilon+C \delta_0\right)\left\|\nabla^{N+1} u\right\|_{L^2}^2
+C_\epsilon(1+t)^{-1}
\left\|\nabla^N(u,\mathbb{E})\right\|_{L^2}^2
+C_\epsilon(1+t)^{-(N+1+s)}.
\end{align}
A further application of Lemma \ref{A7-commutator estimates-202512222029} and \eqref{pure-energy-202512221703} yields
\begin{align}\label{2314}
\int \nabla^N g_2 \cdot \nabla^N \mathbb{E} d x
\leq
\epsilon\left\|\nabla^{N+1} u\right\|_{L^2}^2
+C_\epsilon(1+t)^{-1}
\left\|\nabla^N(u,\mathbb{E})\right\|_{L^2}^2,
\end{align}
Inserting \eqref{2313} and \eqref{2314} into \eqref{do} and choosing $\epsilon$, $\delta_0$ sufficiently small yields
\begin{align}\label{2316}
\frac{d}{d t}
\left\|\nabla^N(u,\mathbb{E})\right\|_{L^2}^2
+\left\|\nabla^{N+1} u\right\|_{L^2}^2
\leq C(1+t)^{-1}
\left\|\nabla^N(u,\mathbb{E})\right\|_{L^2}^2
+C(1+t)^{-(N+1+s)}.
\end{align}
Multiplying the above inequality by $(1+t)^{N+s+\epsilon_0}$ and integrating in time, we conclude
\begin{align}\label{2317}
\nonumber
&
(1+t)^{N+s+\epsilon_0}
\left\|\nabla^N(u,\mathbb{E})\right\|_{L^2}^2
+\int_0^t(1+\tau)^{N+s+\epsilon_0}
\left\|\nabla^{N+1} u\right\|_{L^2}^2 d \tau \\
\leq
\nonumber
&
\left\|\nabla^N(u_0,\mathbb{E}_0)\right\|_{L^2}^2
+C \int_0^t(1+\tau)^{N+s-1+\epsilon_0}
\left\|\nabla^N(u,\mathbb{E})\right\|_{L^2}^2 d \tau
+C \int_0^t(1+\tau)^{-1+\epsilon_0} d \tau \\
\leq
&
C\left(1+\left\|\nabla^N(u_0,\mathbb{E}_0)\right\|_{L^2}^2\right)
+C(1+t)^{\epsilon_0},
\end{align}
where the decay estimate \eqref{pure-energy-202512221703} was used. Consequently, \eqref{2317} directly yields \eqref{dissipative1can}, completing the proof of the lemma.
\end{proof}
\noindent{\bf The proof of Theorem \ref{pure-energy-202512221702} \eqref{N-pure-energy-202512221703}}
From \eqref{dissipative1can} (Lemma \ref{dissipativecan}), we obtain
\begin{align*}
\left\|\nabla^N(\rho-\bar{\rho},u,\mathbb{F}-\mathbb{I},\nabla \Phi)(t)\right\|_{L^2}^2
\leq C(1+t)^{-(s+N)},
\end{align*}
where $C$ is a positive constant independent of time. Thus we immediately obtain estimate \eqref{N-pure-energy-202512221703}, which completes the proof of Theorem \ref{pure-energy-202512221702}.

\appendix
\renewcommand{\appendixname}{Appendix~\Alph{section}}

%%%%%%%%%%%%%%%%%%%%%%%%%%%%%%%%%%%%%%%%%%%%%%%
\section{Auxiliary lemmas}\label{appendix}
%%%%%%%%%%%%%%%%%%%%%%%%%%%%%%%%%%%%%%%%%%%%%%%
We begin by stating some auxiliary lemmas that will be used extensively in this paper. Foremost among them is the Gagliardo-Nirenberg-Sobolev inequality.
\begin{Lemma}\label{A1-N-202512221945}
Let $0 \leq \gamma, \alpha \leq \beta$. Then we have
\begin{align*}
\left\|\nabla^\gamma f\right\|_{L^p} \lesssim\left\|\nabla^\alpha f\right\|_{L^q}^{1-\sigma}\left\|\nabla^\beta f\right\|_{L^r}^\sigma,
\end{align*}
where $0 \leq \sigma \leq 1$ and $\gamma$ satisfies
\begin{align*}
\frac{\gamma}{3}-\frac{1}{p}=\left(\frac{\alpha}{3}-\frac{1}{q}\right)(1-\sigma)+\left(\frac{\beta}{3}-\frac{1}{r}\right) \sigma .
\end{align*}
Here, when $p=\infty$, we require that $0<\sigma<1$.
\end{Lemma}
\begin{proof}
We refer to Theorem 1.1 (p. 125) of \cite{Nirenberg1959}.
\end{proof}

For $s \in\left[0, \frac{3}{2}\right)$, the Hardy-Littlewood-Sobolev theorem gives rise to an inequality of $L^p$ type.

\begin{Lemma}\label{A2-HL-202512222020}
Let $0 \leq s<\frac{3}{2}$ and $\frac{1}{2}+\frac{s}{3}=\frac{1}{p}$. Then $1<p \leq 2$ and
\begin{align*}
\|f\|_{\dot{H}^{-s}} \lesssim\|f\|_{L^p}.
\end{align*}
\end{Lemma}
\begin{proof}
See Theorem 1 (p. 119) in \cite{Stein1970}.
\end{proof}

For homogeneous Besov spaces with regularity index $s \in\left(0, \frac{3}{2}\right]$, a similar $L^p$-type inequality admits.

\begin{Lemma}\label{A3-B-202512222024}
Let $0<s \leq \frac{3}{2}$ and $\frac{1}{2}+\frac{s}{3}=\frac{1}{p}$. Then $1 \leq p<2$ and
\begin{align*}
\|f\|_{\dot{B}_{2, \infty}^{-8}} \lesssim\|f\|_{L^p} .
\end{align*}
\end{Lemma}
\begin{proof}
See Lemma 4.6 in \cite{Guo-Wang2012}.
\end{proof}

To this end, we make use of a specialized Sobolev interpolation:

\begin{Lemma}\label{A4-special Sobolev interpolation-202512222025}
Let $s \geq 0$ and $l \geq 0$. Then we have
\begin{align*}
\left\|\nabla^l f\right\|_{L^2} \leq\left\|\nabla^{l+1} f\right\|_{L^2}^{1-\theta}\|f\|_{\dot{H}^{-s}}^\theta,
\end{align*}
where $\theta=\frac{1}{l+1+s}$.
\end{Lemma}
\begin{proof}
This is a direct consequence of an application of the Parseval and Holder inequalities.
\end{proof}
We will employ the following special interpolation within the Besov space framework:
\begin{Lemma}\label{A5-special Besov interpolation-202512222027}
Let $s>0$ and $l \geq 0$. Then we have
\begin{align*}
\left\|\nabla^l f\right\|_{L^2} \leq\left\|\nabla^{l+1} f\right\|_{L^2}^{1-\theta}\|f\|_{\dot{B}_{2, \infty}^{-s}}^\theta,
\end{align*}
where $\theta=\frac{1}{l+1+s}$.
\end{Lemma}
\begin{proof}
We refer the reader to Lemma 4.5 in \cite{Sohinger-Strain2014}.
\end{proof}

We recall that the commutator estimates are as follows:

\begin{Lemma}\label{A7-commutator estimates-202512222029}
Let $m \geq 1$ be an integer and define the commutator
\begin{align*}
\left[\nabla^m, f\right] g=\nabla^m(f g)-f \nabla^m g .
\end{align*}
Then we have
\begin{align*}
\left\|\left[\nabla^m, f\right] g\right\|_{L^p} \lesssim\|\nabla f\|_{L^{p_1}}\left\|\nabla^{m-1} g\right\|_{L^{p_2}}+\left\|\nabla^m f\right\|_{L^{p_3}}\|g\|_{L^{p_4}},
\end{align*}
where $p, p_1, p_2, p_3, p_4 \in[1,+\infty]$ and
\begin{align*}
\frac{1}{p}=\frac{1}{p_1}+\frac{1}{p_2}=\frac{1}{p_3}+\frac{1}{p_4} .
\end{align*}
\end{Lemma}
\begin{proof}
We refer to Lemma 3.4 (p. 129) in \cite{Majda-Bertozzi2002}.
\end{proof}

Two important identities are noted, finally.

\begin{Lemma}\label{A8-two important identities-202512222137}
Assume that the initial conditions \eqref{202512231806-initial-data-conditions} are satisfied. Then we have for all $t \geq 0$,
\begin{align}\label{A1-A8-two important identities-202512222139}
\operatorname{div}\left( \mathbb{F}^T\right)(t)=0
\end{align}
and
\begin{align}\label{A2-A8-two important identities-202512222140}
\mathbb{F}^{l k}(t) \nabla_l \mathbb{F}^{i j}(t)=\mathbb{F}^{l j}(t) \nabla_l \mathbb{F}^{i k}(t),
\end{align}
where the Einstein's summation convention is used.
\end{Lemma}
\begin{proof}
See \cite{Lin-Liu-Zhang2005} \cite{Liu-Walkington2001}.
\end{proof}
We present two inequalities frequently employed in the main text.
\begin{Lemma}\label{ele}
The following estimate holds whenever $\max \{a, b\}>1$:
\begin{align*}
\int_0^t(1+t-\tau)^{-a}(1+\tau)^{-b} \mathrm{~d} \tau \leq C(1+t)^{-\min \{a, b\}}, \quad t \geq 0.
\end{align*}
\end{Lemma}
\begin{proof}
We refer the reader to \cite{Duan-Ukai-Yang-Zhao2007}.
\end{proof}
\begin{Lemma}\label{sbo}
We have the following inequalities:
\begin{align*}
(i)\|u\|_{L^p} \leq C\|u\|_{H^1},
\quad
2 \leq p \leq 6,
\quad
 u \in H^1;\\
(ii)\|u\|_{L^p} \leq C\|u\|_{H^2},
\quad
2 \leq p \leq \infty,
\quad
 u \in H^2.
\end{align*}
\end{Lemma}

\section*{Acknowledgements}

This work was partially supported by National Key R\&D Program of China (No. 2021YFA1002900) and Guangzhou City Basic and Applied Basic Research Fund (No. 2024A04J6336).

\bigskip

{\bf Data Availability:} No data was used for the research described in the article.

\bigskip

{\bf Conflict of Interest:} The authors declare that they have no conflict of interest.

%%\begin{thebibliography}{99}
%%\end{thebibliography}

%\bibliography{bib}

\end{document}